\documentclass[11pt,english]{article}%

\usepackage[letterpaper,top=2cm,bottom=2.5cm,left=3cm,right=3cm,marginparwidth=1.75cm]{geometry}

\usepackage{authblk}
\usepackage{orcidlink}

\usepackage[utf8]{inputenc}
\usepackage[english]{babel}
\usepackage{amsmath}
\usepackage{algorithmic}
\usepackage{amssymb}
\usepackage{amsthm}
\usepackage{graphicx}
\usepackage{color}
\usepackage{tikz}
\usetikzlibrary{calc,angles,positioning}
\usepackage{pgfplots}
\pgfplotsset{compat=1.11}
\usepackage{comment}

\theoremstyle{definition}
\numberwithin{equation}{section}

\newtheorem{theorem}{Theorem}[section] 
\newtheorem{lemma}[theorem]{Lemma}
\newtheorem{corollary}[theorem]{Corollary}

\newtheorem{thmy}{Theorem}

\theoremstyle{definition}
\newtheorem{proposition}[theorem]{Proposition}

\newtheorem{definition}[theorem]{Definition}
\newtheorem{remark}[theorem]{Remark}
\newtheorem{example}[theorem]{Example}

\newcommand{\bbD}{{\mathbb D}}
\newcommand{\bbZ}{\mathbb{Z}}
\newcommand{\bbC}{\mathbb{C}}
\newcommand{\bbR}{\mathbb{R}}

\newcommand{\bbN}{\mathbb{N}}

\newcommand{\Cc}{\widehat{\bbC}}

\newcommand{\tzo}{\theta_{01}}
\newcommand{\toz}{\theta_{10}}

\newcommand{\M}{\mathcal M}

\providecommand{\keywords}[1]
{
  \small	
  \textbf{\textit{Keywords:}} #1
}

\begin{document}


\title{Sturmian external angles of primitive components in the Mandelbrot set}

\author[1]{{Benjam\'in A. Itz\'a-Ortiz}{\orcidlink{0000-0002-6189-3266}}}
\author[2]{{M\'onica Moreno Rocha}\orcidlink{0000-0003-3816-4425}}
\author[3]{{V\'ictor Nopal-Coello}\orcidlink{0000-0003-2608-3636}}

\affil[1]{\'Area Acad\'emica de Matem\'aticas y F\'isica, Universidad Aut\'onoma del Estado de Hidalgo, Pachuca, 42184, Hidalgo, M\'exico\\ (itza@uaeh.edu.mx)}

\affil[2]{Centro de Investigaci\'on en Matem\'aticas, Guanajuato, 36023, Guanajuato,  M\'exico\\ (mmoreno@cimat.mx)}

\affil[3]{Centro de Investigaci\'on en Matem\'aticas, Mérida, 97302, Yucatán,  M\'exico\\ (victor.nopal@cimat.mx)}

\maketitle


\begin{abstract}
In this work we introduce the \emph{broken line} construction, which is a geometric and combinatorial algorithm that computes periodic Sturmian angles of a given period, yielding the locations of their landing parameters in the Mandelbrot set. An easy to implement method to compute the conjugated angle of a periodic Sturmian angle is also provided. Furthermore, if $\theta$ is a periodic Sturmian angle computed by the broken line construction, then we show the existence of a one-to-one correspondence between its binary expansion and its associated kneading sequence.
\end{abstract}



\section{Introduction}

Consider the one-parameter family of quadratic polynomials $f_c(z)=z^2+c$ with $z,c\in\bbC$, and denote by $\M$ the Mandelbrot set, which is the set of $c$-parameters for which the forward orbit of the origin under $f_c$ remains bounded from infinity. As established in \cite{DH1, DH2}, the Mandelbrot set is connected and $\Cc\setminus \mathcal M$ is simply connected, thus admitting a conformal isomorphism $\Phi:\Cc\setminus \M \to \Cc\setminus \overline{\bbD}$ that fixes infinity and is tangent to the identity. The \emph{(parametric) external ray of external angle} $\theta\in \bbR/\bbZ$, is the set
\[R_\theta=\left\{\Phi^{-1}\left(t e^{2\pi i \theta}\right)~:~t\in (1,\infty]\right\}.\]
The ray $R_\theta$ \emph{lands} at $c\in \partial \M$ if  $\lim_{t\searrow 1} \Phi^{-1}(t e^{2\pi i \theta})=c$. If $\theta$ is rational, then $R_\theta$ lands. Moreover, if $\theta$ is periodic under the doubling map $D(\theta)= 2\theta\mod 1$, then there is a unique angle, $\theta'$ whose associated external ray shares the same landing point of $R_\theta$, in this case, the angles $\theta$ and $\theta'$ are said to be \emph{conjugated}. Both conjugated angles have periodic binary expansions. We say that $\theta$ is a \emph{periodic Sturmian angle} if its binary expansion satisfies the Sturmian condition: namely if the number of $1$s (and hence, the number of $0$s) in two words of same length differ by at most one. For example, the smallest external angle associated with the kokopelli component is $\frac{4}{15}=0.\overline{0100}$, which satisfies the Sturmian condition, whereas its conjugate angle $\frac{3}{15}=0.\overline{0011}$ does not.

\subsection{Notation and terminology}

Throughout this article, we work with rational numbers in $[0,1]$ written in their reduced form. Given any two finite words $a=\alpha_1\ldots \alpha_k$ and $b=\beta_1\ldots \beta_m$, denote the \emph{concatenation} of both words by $a b=\alpha_1\ldots \alpha_k \beta_1\ldots \beta_m$. The term $a^n$ represents for the concatenation of the word $a$ with itself $n$-times. An infinite periodic sequence of period $k\geq 1$ is denoted by $a=\overline{\alpha_1\ldots \alpha_k}$ with \emph{associated word} $w=\alpha_1\ldots \alpha_k$, thus $a=\overline{w}$. The $k$-periodic \emph{binary expansion} of a rational number $p/q\in[0,1]$ is denoted as $p/q=0.\overline{\alpha_1 \alpha_2\ldots \alpha_k}$, with $\alpha_i\in \{0,1\}$.

A \emph{hyperbolic component} of period $m\geq 1$ of the Mandelbrot set is an open and connected set $H\subset \mathcal M$ consisting of parameter values $c$ for which $f_c$ has a unique attracting cycle $A_c=\{z_1(c), \ldots, z_m(c)\}$ of minimal period $m$. There exists a biholomorphism $\lambda_H: H\to \bbD$ which assigns to each $c\in H$ the multiplier of the cycle $A_c$; moreover $\lambda_H$ has a homeomorphic extension to the boundaries, thus providing a parametrization of $\partial H$. For example, the boundary of the unique hyperbolic component of period 1 in $\mathcal M$, denoted by $H_\heartsuit$, is well known to be a cardioid curve.

The parameter $c_H=\lambda_{H}^{-1}(1)$ is called the \emph{root} of $H$. Similarly, the parameter $c_{p/q}\in \partial H$ that satisfies $\lambda_H(c_{p/q})=\exp(2i\pi p/q)$, is the root of a new hyperbolic component $H'$ of period $m q$ associated with $H$, and thus it is called a \emph{satellite} component of $H$. We will write $H'=H\ast H(p/q)$ to denote the satellite component of $H$ at \emph{internal angle} $p/q$. In particular, if $H=H_\heartsuit$, we simply write $H'=H(p/q)$ and call it the \emph{$p/q$-bulb} of the Mandelbrot set. If $H$ is a hyperbolic component of period $m\geq 2$ whose boundary has a cusp at its root point, then we call $H$ a \emph{primitive} component (dynamically speaking, $f_{c_H}$ has a unique parabolic cycle of period $m$ and the immediate parabolic basin of each point in the cycle contains a single attracting petal). 

The root parameter of a hyperbolic component $H\neq H_\heartsuit$ is the landing point of exactly two parametric external rays of angles $0<\theta<\theta'<1$, we say that $\{\theta,\theta'\}$ is the \emph{characteristic pair} of $H$ and that $\theta$ and $\theta'$ are conjugated angles. Moreover, if $H$ is a hyperbolic component of period $m\geq 2$ then its characteristic angles are also $m$-periodic under the doubling mapping. The external rays that land at the root of a satellite component $H$ divide the plane into two topological disks, say $S_\heartsuit \sqcup S_H$, where $H_\heartsuit\subset S_\heartsuit$ and $H\subset S_H$. If $H=H(p/q)$ is a $p/q$-bulb, then the \emph{$\frac pq$-limb} is the connected component $\mathcal M\cap S_H$. Similarly, if $H'=H(p/q)\ast H(a/b)$ then the set $\mathcal M \cap S_{H'}$ is called the \emph{$\frac ab$-sublimb} of the $\frac pq$-limb. 

\subsection{Statement of results}

There is a long list of algorithms that compute the external angles of the Mandelbrot set, such as \cite{Dou}, \cite{L},  \cite{A}, \cite{MR1269932}, \cite{BS},  and \cite{DM}. Our broken line construction produces binary expansions of periodic Sturmian angles and relies on \emph{cutting} and \emph{mechanical sequences}. These sequences have been extensively studied in different contexts, such as Series \cite{MR810563} and the classical construction in Morse \& Hedlund \cite{MR745}; see also Itz\'a-Ortiz et al. \cite{IO} and their references therein. Connections between Sturmian sequences and the Mandelbrot set can be found in \cite{Kel}.

Given a rational number $0<p/q<1$, consider the straight line $L: y=\frac{p}{q}x$ over $\bbR^2$ together with its integer grid $\bbZ\times\bbZ\subset \bbR^2$. Over the first quadrant, the line $L$ cuts the grid in an infinite collection of points with at least one nonzero integer coordinate, label such points as $(x_i,y_i)\in L$ for $i\in \bbN$. Since the slope is rational, the points $(x_i,y_i)$ determine a periodic sequence of \emph{cutting places}, $k_1 k_2\ldots$ with $k_i\in \{0,1,*\}$, defined as follows:
\begin{equation*}
  k_i = \begin{cases}
  	0 & \text{if } x_i\in \bbN, y_i\notin \bbN,\\
	1 & \text{if } x_i\notin \bbN, y_i\in \bbN,\\
	* & \text{if } (x_i,y_i)=(jq,jp)~\text{for some }j\in \bbN,
  \end{cases} 
\end{equation*}
where the symbol $*$ stands for either $01$ or $10$. Whenever $*=01$ (or $*=10$) we say we have followed the \emph{$01$-convention} (or the \emph{$10$-convention}, respectively).

\begin{definition}[Cutting and mechanical sequences] \label{def:cs}
Given any rational number $0<p/q<1$, the line $L:y=\frac{p}{q}x$ defines two \emph{cutting sequences} that correspond to distinct binary periodic sequences of period $p+q$ given by
\[\kappa_{01}(p/q): = \overline{k_1\ldots k_{p+q-2} 01},\qquad \text{and}\qquad \kappa_{10}(p/q) := \overline{k_1\ldots k_{p+q-2} 10}.\]
We define the \emph{01-} and the \emph{10-Mechanical sequences} of $L$ (or simply, \emph{M-sequences}) as
\begin{equation}\label{eq:lucs}
\tzo(p/q):=0.T(\kappa_{01}(p/q))\qquad \text{and}\qquad \toz(p/q):=0.T(\kappa_{10}(p/q)),
\end{equation}
where $T$ represents for the substitution rule $01\mapsto 1$. 

If $\theta_*(p/q)=0.\overline{\alpha_1\ldots \alpha_q}$ with $*\in\{01,10\}$, we denote its associated word by
\begin{equation}\label{eq:words}
W^*_{p/q}:=\alpha_1\ldots\alpha_q,\qquad\text{so that}\qquad\theta_*(p/q)=0.\overline{W^*_{p/q}}.
\end{equation}
For simplicity, we omit the superscripts $01$ and $10$ from \eqref{eq:words} whenever the convention employed is obvious (for example, in $\toz(p/q)=0.\overline{W_{p/q}}$). Whenever we need to make an explicit mention of the convention, we will preserve the superscripts. 
\end{definition}

\begin{example}\label{ex:Ans}
Let $m\geq 2$ and consider the graph of the line $L: y=\frac 1m x$ restricted to the first quadrant. $L$ cuts the integer grid at points of the form $(x_i,y_i)=(i,\frac im)$ for $i=1,\ldots,m$, thus
\begin{equation*}
\kappa_{01}(1/m) = \overline{0^{m-1} 01} \qquad \text{and} \qquad \kappa_{10}(1/m)= \overline{0^{m-1} 10}.
\end{equation*}
Applying the substitution rule $T: 01\mapsto 1$ to both cutting sequences yields
\begin{equation}\label{eq:An}
\tzo(1/m)=0.\overline{0^{m-1}1},\qquad\text{and}\qquad \toz(1/m)=0. \overline{0^{m-2}10}.
\end{equation}
\end{example}

In the example, the M-sequences obtained coincide with the binary expansions of the characteristic external angles landing at the root of the bulb $H(1/m)$. This is true in general as shown in Proposition~\ref{prop:ExtAngles}. The next result shows that periodic Sturmian angles in the Mandelbrot set are ``rare''.

\begin{proposition}
For any given angle $\varphi\in \bbR/\bbZ, \varphi \neq 0 \mod 1$ whose binary expansion contains the words $01$ and $10$, then the angle resulting from tuning $\varphi$ with the characteristic angles of any bulb, has a non-Sturmian binary expansion. In particular, if $H$ is any satellite component that is not a bulb, then its characteristic angles are not periodic Sturmian angles.
\end{proposition}

\begin{proof}
Let $\{\theta^-, \theta^+\}$ be the characteristic angles associated with the principal bulb $H(p/q)$, with $(p,q)=1$ and $q\geq 2$. Proposition~\ref{prop:ExtAngles} shows that $\theta^-=\tzo(p/q)=\overline{\alpha_1 \ldots \alpha_{q-2} 01}$ while $\theta^+=\toz(p/q)=\overline{\alpha_1 \ldots \alpha_{q-2} 10}$. By hypothesis $\varphi$ contains the words $01$ and $10$, hence, Douady's tuning procedure implies that the new angle $\varphi \ast H(p/q)$ contains the words
\[ \alpha_1 \ldots \alpha_{q-2} 01~\alpha_1 \ldots \alpha_{q-2} 10,\qquad \text{therefore containing}\qquad 1\alpha_1\ldots \alpha_{q-2} 1,\]
and
\[ \alpha_1 \ldots \alpha_{q-2} 10~\alpha_1 \ldots \alpha_{q-2} 01,\qquad \text{therefore containing}\qquad 0\alpha_1\ldots \alpha_{q-2} 0.\]
Clearly, if the 1-frequency in $1\cdot \alpha_1\ldots \alpha_{q-2} 1$ is $n$, then the 1-frequency in $0\cdot \alpha_1\ldots \alpha_{q-2} 0$ is $n-2$, implying that  $\varphi \ast H(p/q)$ is not Sturmian.
\end{proof}

\begin{remark}
As pointed out in \cite{MR1269932} (see also \cite{MR3793658}), Gambaudo et al. have shown in \cite{MR0772104} that the Sturmian condition imposed on the binary expansion of $\theta\in \bbR/\bbZ$ is equivalent to the preservation of cyclic order by the doubling map over $\theta$. Therefore, if $\theta$ is a periodic Sturmian angle then its orbit under the doubling map is a \emph{rotation set} for some given rational rotation number $a/b\in \bbR/\bbZ$. A straightforward computation shows that for every $b\geq 2$, the number of periodic Sturmian angles landing at primitive components of period $b$ are exactly $(b-2)\varphi(b)$, where $\varphi$ is Euler's totient function.
\end{remark}

In this work, we show that considering the graphs of piecewise linear maps and their associated cutting and mechanical sequences, we can compute the periodic Sturmian angles associated with primitive hyperbolic components in $\mathcal M$.  To that end, we set the following notation: 
\begin{definition}(Hinge points and broken lines)\label{def:BL}
For a given line $L$ with slope $p/q$ and for any integer $n\geq 1$, we say that the point $(nq,np)\in L$ is the  \emph{$n^{\text{th}}$ hinge point} of $L$. We define the \emph{broken line} at the $n^{\text{th}}$ hinge point, denoted by $BL(\frac pq,\frac a b,n)$, as the piecewise linear map given by
\begin{equation}\label{eq:BL}
  y = \begin{cases}
  	\frac{p}{q}x & \text{if } 0\leq x\leq nq,\\
	\frac{a}{b}(x-nq)+np & \text{if } nq\leq x <\infty,
  \end{cases} 
\end{equation}
with $a/b\in ]0,1[$. 
\end{definition}

The infinite collection of cutting points of the graph of the broken line with the integer lattice in the first quadrant defines the cutting sequences with respect to the selected convention. After applying the substitution rule, we denote by $\tzo(\frac pq,\frac a b,n)$ and $\toz(\frac pq,\frac a b,n)$ the respective M-sequences associated with the broken line $BL(\frac pq,\frac a b,n)$. The next result is a direct consequence of the geometry of the broken line (see Figure~\ref{AB}).

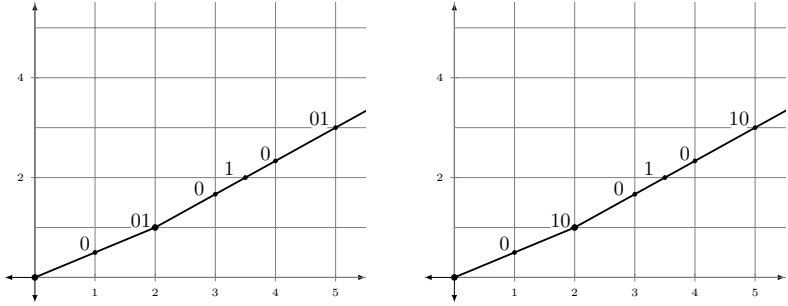
\begin{figure}[h]
\centering
\begin{tabular}{c @{\qquad} c }
\begin{tikzpicture}[thick, scale=0.7]
\begin{axis}[
    xmin=0,xmax=5,ymin=0,ymax=5,
    axis lines= middle,
    enlargelimits={abs=0.5},
    axis line style={latex-latex},
    ticklabel style={font=\tiny,fill=white},
    every axis x label/.style={at={(current axis.right of origin)},anchor=north west},
    ylabel style={at={(ticklabel* cs:1)},anchor=south west}
        ]

\draw[
  help lines,
  line width=0.1pt,
  gray,
  shift={($(1, 1) - (0, 0)$)},
      ] (-1, -1) grid[step={($(1, 1) - (0, 0)$)}] (5, 5);

\coordinate (O) at (0,0);
\coordinate (A) at (-2,-1);
\coordinate (B) at (2,1);
\node[fill=white,circle,inner sep=0pt] (1.4,0.6) at ($(1.1,0.5)+(150:10pt)$) {\color{black} 0};
\draw[fill] (1,1/2) circle (1pt);

\node[fill=white,circle,inner sep=0pt] (2,1.1) at ($(2,1-0.1)+(140:10pt)$) {\color{black}01};

\node[fill=white,circle,inner sep=0pt] (3,1.7) at ($(3,1.6)+(150:10pt)$) {\color{black}0};
\draw[fill] (3,1+2/3) circle (1pt);

\node[fill=white,circle,inner sep=0pt] (4,2) at ($(4.1,2.3)+(150:10pt)$) {\color{black}0};
\draw[fill] (4,1+4/3) circle (1pt);

\node[fill=white,circle,inner sep=0pt] (3,1.8) at ($(3.5,2)+(150:10pt)$) {\color{black}1};
\draw[fill] (2+3/2,2) circle (1pt);

\node[fill=white,circle,inner sep=0pt] (5,3.1) at ($(5,3)+(150:10pt)$) {\color{black}01};
\draw[fill] (7,2+6/7) circle (1pt);
\node[fill=white,circle,inner sep=0pt] (7,2+6/7) at ($(7,2+6/7)+(150:10pt)$) {0};

\node[fill=white,circle,inner sep=0pt] (6,2+3/7) at ($(7.5,3.1)+(150:10pt)$) {1};
\draw[fill] (5+7/3,3) circle (1pt);

\node[fill=white,circle,inner sep=0pt] (6,2+3/7) at ($(6,2+3/7)+(150:10pt)$) {0};
\draw[fill] (6,2+3/7) circle (1pt);

\node[fill=white,circle,inner sep=0pt] (8,17/5) at ($(8,17/5-0.1)+(150:10pt)$) {0};
\draw[fill] (8,2+9/7) circle (1pt);

\node[fill=white,circle,inner sep=0pt] (4,1.8) at ($(9.8,4)+(150:10pt)$) {1};
\draw[fill] (5+14/3,4) circle (1pt);

\node[fill=white,circle,inner sep=0pt] (10,4) at ($(10.4,4.4)+(180:10pt)$) {0};
\draw[fill] (10,2+15/7) circle (1pt);
\node[fill=white,circle,inner sep=0pt] (10.1,5.2) at ($(12.1,5.2)+(180:10pt)$) {01};

\draw[fill] (11,2+18/7) circle (1pt);
\node[fill=white,circle,inner sep=0pt] (8.1,4.1) at ($(11.1,4.8)+(180:10pt)$) {0};

\draw[fill] (9,2+12/7) circle (1pt);
\node[fill=white,circle,inner sep=0pt] (9.1,4.2) at ($(9.1,3.9)+(180:10pt)$) {0};

\draw[fill, black] (O) circle (1.5pt);
\draw[fill] (A) circle (1.5pt);
\draw[fill, black] (B) circle (1.5pt);

\coordinate (A3) at (8,5);
\draw[fill] (5,3) circle (1pt);
\draw[line width=1pt]  (O) -- (B) -- (A3);

\end{axis}

\end{tikzpicture} &

\begin{tikzpicture}[thick, scale=0.7]
\begin{axis}[
    xmin=0,xmax=5,ymin=0,ymax=5,
    axis lines= middle,
    enlargelimits={abs=0.5},
    axis line style={latex-latex},
    ticklabel style={font=\tiny,fill=white},
    every axis x label/.style={at={(current axis.right of origin)},anchor=north west},
    ylabel style={at={(ticklabel* cs:1)},anchor=south west}
        ]

\draw[
  help lines,
  line width=0.1pt,
  gray,
  shift={($(1, 1) - (0, 0)$)},
      ] (-1, -1) grid[step={($(1, 1) - (0, 0)$)}] (5, 5);

\coordinate (O) at (0,0);
\coordinate (A) at (-2,-1);
\coordinate (B) at (2,1);
\node[fill=white,circle,inner sep=0pt] (1.4,0.6) at ($(1.1,0.5)+(150:10pt)$) {\color{black} 0};
\draw[fill] (1,1/2) circle (1pt);

\node[fill=white,circle,inner sep=0pt] (2,1.1) at ($(2,1-0.1)+(140:10pt)$) {\color{black}10};

\node[fill=white,circle,inner sep=0pt] (3,1.7) at ($(3,1.6)+(150:10pt)$) {\color{black}0};
\draw[fill] (3,1+2/3) circle (1pt);

\node[fill=white,circle,inner sep=0pt] (4,2) at ($(4.1,2.3)+(150:10pt)$) {\color{black}0};
\draw[fill] (4,1+4/3) circle (1pt);

\node[fill=white,circle,inner sep=0pt] (3,1.8) at ($(3.5,2)+(150:10pt)$) {\color{black}1};
\draw[fill] (2+3/2,2) circle (1pt);

\node[fill=white,circle,inner sep=0pt] (5,3.1) at ($(5,3)+(150:10pt)$) {\color{black}10};
\draw[fill] (7,2+6/7) circle (1pt);
\node[fill=white,circle,inner sep=0pt] (7,2+6/7) at ($(7,2+6/7)+(150:10pt)$) {0};

\node[fill=white,circle,inner sep=0pt] (6,2+3/7) at ($(7.5,3.1)+(150:10pt)$) {1};
\draw[fill] (5+7/3,3) circle (1pt);

\node[fill=white,circle,inner sep=0pt] (6,2+3/7) at ($(6,2+3/7)+(150:10pt)$) {0};
\draw[fill] (6,2+3/7) circle (1pt);

\node[fill=white,circle,inner sep=0pt] (8,17/5) at ($(8,17/5-0.1)+(150:10pt)$) {0};
\draw[fill] (8,2+9/7) circle (1pt);

\node[fill=white,circle,inner sep=0pt] (4,1.8) at ($(9.8,4)+(150:10pt)$) {1};
\draw[fill] (5+14/3,4) circle (1pt);

\node[fill=white,circle,inner sep=0pt] (10,4) at ($(10.4,4.4)+(180:10pt)$) {0};
\draw[fill] (10,2+15/7) circle (1pt);
\node[fill=white,circle,inner sep=0pt] (10.1,5.2) at ($(12.1,5.2)+(180:10pt)$) {10};

\draw[fill] (11,2+18/7) circle (1pt);
\node[fill=white,circle,inner sep=0pt] (8.1,4.1) at ($(11.1,4.8)+(180:10pt)$) {0};

\draw[fill] (9,2+12/7) circle (1pt);
\node[fill=white,circle,inner sep=0pt] (9.1,4.2) at ($(9.1,3.9)+(180:10pt)$) {0};

\draw[fill, black] (O) circle (1.5pt);
\draw[fill] (A) circle (1.5pt);
\draw[fill, black] (B) circle (1.5pt);

\coordinate (A3) at (8,5);
\draw[fill] (5,3) circle (1pt);
\draw[line width=1pt]  (O) -- (B) -- (A3);

\end{axis}

\end{tikzpicture}\\
\end{tabular}
\caption{Geometric representations of $01$-convention (left) and $10$-convention (right) of the broken line $BL\left(\frac{1}{2},\frac{2}{3},1\right)$}
\label{AB}
\end{figure}

\begin{lemma}\label{lem:easybl}
Let $*\in \{01,10\}$. Given two rational numbers $0<p/q,a/b<1$ and their M-sequences  $\theta_*(p/q)=0.\overline{W^*_{p/q}}$ and $\theta_*(a/b)=0.\overline{W^*_{a/b}}$,  then the M-sequence of the broken line at the $n^{\text{th}}$ hinge point, $BL(\frac pq,\frac a b,n)$, is given by the concatenation
\[\theta_*\left(\frac pq,\frac a b,n\right) = 0.(W^*_{p/q})^n\overline{W^*_{a/b}}.\]
\end{lemma}

We now concentrate on M-sequences and broken lines via the $01$-convention. Analogous results for the $10$-convention are discussed in Appendix~\ref{App:A}. Let $n\geq 1$ and set $\theta=\tzo(P/Q, a/b, n)$ with fractions $\frac PQ, \frac ab, \frac AB, \frac ST$ satisfying the \emph{01-hypothesis}: namely $\frac AB, \frac ST$ are Farey neighbors, 
\begin{equation}\label{eq:MainHyp}
0\leq \frac AB<\frac PQ < \frac ab < \frac{S_n}{T_n}\leq \frac ST\leq 1,\quad \frac PQ = \frac{A+S}{B+T}\quad\text{and}\quad \frac{S_n}{T_n}:=\frac{(n-1)P+S}{(n-1)Q+T}. \tag{01-Hyp}
\end{equation}

As shown in Lemma~\ref{lem:persturm}, condition $\frac PQ<\frac ab<\frac{S_n}{T_n}$ is sufficient to guarantee that the M-sequence $\toz( P/Q,a/b,n)$ is periodic of minimal period $b$.

\begin{thmy}(Primitive components)\label{thm:A}
Fix $n\geq 1$ and let $\frac PQ, \frac ab$ be rational numbers satisfying  (\ref{eq:MainHyp}). If $\theta=\tzo(P/Q, a/b, n)$ is the M-sequence of the broken line $BL(\frac{P}{Q},\frac{a}{b},n)$ in the $01$-convention, then the external ray $R_\theta$ lands at the root of a primitive component of period $b$.
\end{thmy}

The strategy behind the proof of Theorem~\ref{thm:A} is as follows: by Lemma~\ref{lem:persturm}, the angle $\theta$ has a periodic binary expansion and therefore it is a rational angle that lands at the root of a hyperbolic component, \cite{DH1, DH2, S}. If $\theta'$ denotes the conjugate angle of $\theta$, then we show that their \emph{kneading sequences} (see definition in Section~\ref{sec:kneading}) have period $b$. Therefore with Corollary 5.5 (Exact Periods of Kneading Sequences) in \cite{LS}, we conclude that $R_\theta$ lands at the root of a primitive component. Section~\ref{sec:Proofs} contains a generalization of the results found in \cite{DM} which yields a precise location of the primitive component in the $\frac PQ$-limb, see Proposition~\ref{prop:location}.

Under the conditions in (\ref{eq:MainHyp}), for $n\geq 1$ and $m\geq 1$ given, define a \emph{block of words}
\begin{equation*}
B_{n,m}:=W_{P/Q}^n(W_{S/T}W_{P/Q}^{n-1})^{m-1}W_{S/T},\qquad B_{n,0}:=W_{P/Q},
\end{equation*}
(the motivation behind this definition is discussed in Remark~\ref{rem:blocknotation}). In addition to the digitwise representation of $\theta=0.\overline{\alpha_1\ldots\alpha_b}$, there are two useful alternate representations: namely the \emph{wordwise} and \emph{blockwise} representations which are denoted by
\[\theta=0.\overline{W_{S/T}W_{\zeta_2}\ldots W_{\zeta_{k-1}} W_{P/Q}}\qquad\text{and}\qquad
\theta=0.\overline{B_{n,m+1}B_{n,m_2}\ldots B_{n,m_{s-1}}B_{n,m}},\]
respectively, where $\zeta_j\in \{S/T, P/Q\}$ and $m_j\in \{m, m+1\}$, for some integers $k,s\geq 1$ and $m\geq 0$. 
In Section 2 we describe how to compare two M-sequences in terms of their digit or block representation, which is a central part for the computation of the conjugate angle of $\theta=\tzo(P/Q, a/b,n)$. More precisely, if \(\theta=\tzo(P/Q, a/b,n)\) and \(\theta'\) is as in Definition~\ref{def:tetas} then we have the following result.

\begin{thmy}(Characteristic angles of a broken line)\label{thm:B}
Fix $n\geq 1$ and consider the fractions $\frac PQ, \frac ab$ satisfying (\ref{eq:MainHyp}). Then $\theta$ and $\theta'$ are conjugate external angles.
\end{thmy}

The proof of this theorem consists of verifying that $\theta'$ satisfies Lemma~13.1 (Conjugate External Angle) and Lemma~13.3 (Conjugate External Angle Algorithm) in \cite{BS}.

Our final result shows how to compute the kneading sequence of  $
\tzo(P/Q, a/b,n)$ from its binary expansion, and  vice versa. The proof of part (b) below provides an algorithm to compute the rational numbers associated with the broken line.

\begin{thmy}(Kneading sequence of a broken line)\label{thm:C}
Fix $n\geq 1$ and let $\frac{P}{Q}, \frac{a}{b},\frac{S}{T}$ be rational numbers satisfying (\ref{eq:MainHyp}). Denote the M-sequence of the broken line $BL(\frac{P}{Q},\frac{a}{b},n)$ by $\theta=0.\overline{\alpha_1\cdots\alpha_b}= 0.\overline{W_{\zeta_1}\ldots W_{\zeta_k}}$, for some $k\geq1$ and $\zeta_i\in\{S/T,P/Q\}$ for all $i=1,\ldots,k$. Let $K(\theta)=\overline{\Omega_1\cdots\Omega_{b-1}*}$ be the kneading sequence associated with $\theta$, with $\Omega_j\in \{\mathnormal{0,1}\}$ for all $i=1,\ldots,b-1$. Then,
\begin{itemize}
    \item[(a)] $\Omega_i=\mathnormal{0}$ if and only if $\alpha_{i+1}$ is the first digit of the string $W_{P/Q}^sW_{S/T}$ for some $0\leq s<n$.
    \item[(b)] The numbers $\frac{P}{Q}$, $\frac{S}{T}$, $\frac{a}{b}$ and the M-sequence of $BL(\frac{P}{Q},\frac{a}{b},n)$ can be recovered from the kneading sequence $K(\theta)$.
\end{itemize}
\end{thmy}

\subsection{Examples}

To illustrate Theorems~\ref{thm:A} and \ref{thm:B} first consider the $\frac 12$-limb. Set $\frac PQ = \frac 12$ with Farey parents $\frac AB=\frac 01$ and $\frac ST = \frac 11$. Select $\frac ab = \frac 34$ so that (\ref{eq:MainHyp}) is satisfied for $n=1$. From Corollary~\ref{cor:Cor1} it follows $\tzo(3/4)=0.\overline{W_{3/4}}=0.\overline{W_1^2W_{1/2}}$. Theorem~\ref{thm:A} shows that the parametric external ray with angle
\[\theta_{01}:=\tzo\left(\frac{1}{2},\frac{3}{4},1\right)=0.W_{1/2}\overline{W_{3/4}}
=0.\overline{W_{1/2}W_{1}W_{1}}=0.\overline{B_{1,2}}=0.\overline{0111}=\frac{7}{15}\]
lands at the cusp of a primitive hyperbolic component of period $4$. To compute its characteristic angle, we recall from Definition~\ref{def:tetas} that $B'_{1,2}=0111+1=1000$, therefore,  $\theta'_{01}:=0.\overline{B'_{1,2}}= 0.\overline{1000}=\frac{8}{15}$ is the characteristic angle of $\theta_{01}$ by~Theorem~\ref{thm:B}. For the $10$-convention, consider again the $\frac 12$-limb and pick $\frac ab=\frac 14$, so (\ref{eq:MainHyp10}) is satisfied for $n=1$. Then $\toz(1/4)=0.\overline{W_{1/4}}=0.\overline{W_0^2W_{1/2}}$ and thus
\[\toz:=\toz\left(\frac{1}{2},\frac{1}{4},1\right)=0.W_{1/2}\overline{W_{1/4}}
=0.\overline{W_{1/2}W_{0}W_{0}}=0.\overline{B_{1,2}}=0.\overline{1000}=\frac{8}{15}.\]

\begin{remark}
Observe that $\frac 8{15}$ is the characteristic angle of the broken line $BL(1/2,3/4,1)$ in the $01$-convention and also the angle of $BL(1/2,1/4,1)$ in the $10$-convention. This is not a coincidence: The Farey diagram for rational numbers in $[0,1]$ visualized in the upper half-plane exhibits a symmetry around $\frac 12$ given by the transformation $\frac ab\mapsto 1-\frac{a}{b}$, thus $\frac 34$ is symmetric to $\frac 14$, and vice versa. This symmetry translates into the $\frac 12$-limb (and only into this limb as a consequence of Proposition~\ref{prop:location}) as follows. \end{remark}

\begin{lemma}
Let $\frac PQ=\frac 12$ and assume $\frac ab$ satisfies (\ref{eq:MainHyp}) for $n=1$. If $\theta=\tzo\left(\frac 12,\frac ab,1\right)$ and $\theta'$ is its characteristic angle, then $\theta'=\toz\left(\frac 12,1-\frac{a}{b},1\right)$.
\end{lemma}

\begin{proof}
Let $m\geq 0$ be given so that $\frac{1+m}{2+m} < \frac ab \leq \frac{2+m}{3+m}$. If the right-hand equality holds, then $\theta=\tzo(1/2,a/b,1)=0.\overline{B_{1,m+1}}$. Otherwise, by Lemma~\ref{lem:R1} there exists $k\geq 2$ and $m_j\in\{m,m+1\}$ with  $j=1,\ldots,k$ so that
\[\theta=\tzo\left( \frac 12, \frac ab,1\right) = 0.\overline{B_{1,m_1} B_{1,m_2}\cdots B_{1,m_{k-1}} B_{1,m_k}}=0.\overline{B_{1,m+1} B_{1,m_2}\cdots B_{1,m_{k-1}} B_{1,m}}.\]
The first case can be easily deduced from the second case, so we work only with the second case.
The hypotheses $n=1, \frac PQ=\frac 12, \frac ST =1$ applied to (\ref{def:D1}) yields $B_{1,m}=W_{P/Q}W_{S/T}^m=01^{m+1}$ and by Definition~\ref{def:tetas} it follows that $B'_{1,m}=10^{m+1}$. Therefore $\theta = 0.\overline{01^{m+2}01^{m_2+1}\cdots 01^{m_{k-1}+1} 01^{m+1}}$ and from Theorem~\ref{thm:B}, it follows that
\[\theta'=0.\overline{B'_{1,m+1}B'_{1,m_2}\cdots B'_{1,m_{k-1}}B'_{1,m}} = 0.\overline{10^{m+2}10^{m_2+1}\cdots 10^{m_{k-1}+1} 10^{m+1}}.\]
Now consider  $0<1-\frac ab<\frac 12$ and the $10$-convention: symmetry with respect to $\frac 12$ implies $\frac 1{3+m} \leq 1-\frac{a}{b}<\frac{1}{2+m}$. Then, Lemma~\ref{rem:R1-10} the shows that $\toz(1/2,1-a/b,1)=0.\overline{B_{1,m+1}}$ or
\[\toz\left( \frac 12, 1-\frac{a}{b},1\right) = 0.\overline{B_{1,m+1} B_{1,m_2}\ldots B_{1,m_{k-1}}, B_{1,m}},\]
where in this case, the blocks $B_{1,m}$ in the $10$-convention given in (\ref{def:D1-10}) become $B_{1,m}=W_{P/Q}W_{A/B}^{m}=10^{m+1}$, so we are done.
\end{proof}

To exemplify different values of the hinge point, consider the $\frac 25$-limb: Set $\frac PQ = \frac 25$ with Farey parents $\frac AB=\frac 13$ and $\frac ST = \frac 12$ and select $\frac ab = \frac{7}{17}$ so (\ref{eq:MainHyp}) is satisfied for $n=3$. According to  Corollary~\ref{cor:Cor1},  $\tzo(7/17)=0.\overline{W_{7/17}}=\overline{W_{1/2}W_{2/5}^3}$. We can compute three pairs of characteristic angles, $\{\theta_j,\theta'_j\}_{j=1}^3$ via Theorem~\ref{thm:B} which are associated with broken lines with the choice of the first, second and third hinge points, respectively. For example, at the second hinge point we have
\begin{align*}
\theta_2&=\tzo\left(\frac{2}{5},\frac{7}{17},2\right)=0.W_{2/5}^2\overline{W_{7/17}}=0.\overline{W_{2/5}W_{2/5}W_{1/2}W_{2/5}}\\
&=0.\overline{B_{2,1}B_{2,0}}=0.\overline{01001010010101001}= \frac{38057}{131071}
\end{align*}
with characteristic angle $\theta'_2=0.\overline{B_{2,1}'B_{2,0}'}=0.\overline{01001010011001010}= \frac{38090}{131071}$.
  Each pair is associated with a primitive component of period $17$ whose location in the $\frac 25$-limb is described in Proposition~\ref{prop:location}. Similar computations can be performed for the $10$-convention, $\frac PQ=\frac 25$ and $\frac ab=\frac 7{18}$ satisfying (\ref{eq:MainHyp10}) for $n=3$ to obtain three new pairs of characteristic external angles $\{\theta_j,\theta'_j\}_{j=4}^6$, whose associated primitive components of period $18$ lie in the $\frac 25$-limb and their locations are described in Proposition~\ref{prop:location-10}. See Figure~\ref{fig:F1}. 

\begin{figure}[h]
\centering
\includegraphics[width=0.5\textwidth]{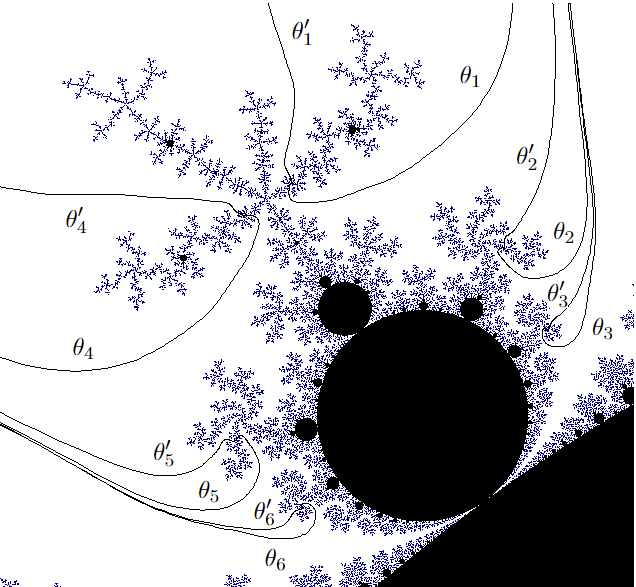}
\caption{Parametric external rays associated with characteristic pairs $\{\theta_j, \theta'_j\}, j=1,\ldots,6$ are displayed. The first three pairs land at cusps of primitive components of period 17 and the other three at cusps for period 18 components, all of which belong to the $\frac 25$-limb. Image generated with Mandel software, by Wolf Jung}
    \label{fig:F1}
\end{figure}

See Example~\ref{ExThmC} to illustrate Theorem~\ref{thm:C}.

\section{Properties of broken lines}\label{sec:PBL}
In this section, we describe how to compute the M-sequences of a line of slope $P/Q\in ]0,1[$ by an appropriate concatenation of the M-sequences of the Farey neighbors of $P/Q$. Once this is performed, we address the same computation for a broken line $BL(\frac PQ, \frac a b,n)$ at a hinge $n\geq 1$ and where $a/b$ is a Farey descendant of $P/Q$ that satisfies the 01-hypothesis (\ref{eq:MainHyp}).

\subsection{M-sequences for straight lines}
Proofs throughout this section will be given only for the $01$-convention and their M-sequences, as the proofs for the $10$-convention and their M-sequences are analogous. 
Recall that for any $p/q\in ]0,1[$ the word $W^*_{p/q}$ is associated with the M-sequence $\theta_*(p/q)$, see (\ref{eq:words}).

\begin{remark}
\label{RemCMS}
By construction, the number of zeros (and consequently ones) in the cutting sequences $\kappa_{01}(p/q)$ and $\kappa_{10}(p/q)$ are  exactly $q$ (consequently $p$), therefore the mechanical sequences $\tzo(p/q)$ and $\toz(p/q)$ contain exactly $q-p$ zeros (and consequently $p$ ones). Moreover, in the $01$-convention, $\tzo(p/q)=0.\overline{\alpha_1\ldots\alpha_q}$, where
 \begin{itemize}
     \item $\alpha_j=0$ if and only if 
     \begin{equation}
     \label{Eq1}
         N<\frac{p}{q}\cdot j<\frac{p}{q}\cdot(j+1)\leq N+1,
     \end{equation}
     for some $N\geq0$. In this case $k_{N+j}k_{N+j+1}=00$ and the substitution rule $T$ maps $k_{N+j}\mapsto\alpha_j$. The equality only occurs when $j=q-1$. See the left diagram in Figure~\ref{fig2.1}.

\begin{figure}
    \centering
    \label{fig2.1}
    \begin{tabular}{c @{\qquad} c }
    \begin{tikzpicture}[thick, scale=0.8]
    \begin{axis}[
    xmin=-0.5,xmax=5,ymin=0,ymax=5,
    axis lines= middle,
    enlargelimits={abs=1},
    axis line style={latex-stealth}, 
    yticklabels={,,},
    xticklabels={,,},
    ticklabel style={font=\tiny,fill=white},
   tick style={draw=none},
   title={Case $\alpha_j=0$}
        ]

    \draw[thick,gray] (0,2) -- (5,2); 
    \draw[thick,gray] (2.2,0) -- (2.2,5); 
    \draw[thick,gray] (0,4) -- (5,4);  
    \draw[thick,gray] (4.2,0) -- (4.2,5); 
    \draw[thick,dotted] (0,0.9) -- (5,3.7); 
    \draw[thick,black] (0,1.648) -- (5,4.448); 

\node[below] at (2.2,0) {$j$};
\node[below] at (4.2,0) {$j+1$};
\node[left] at (0,2) {$N$};
\node[left] at (0,4) {$N+1$};

\end{axis}
\end{tikzpicture} &

\begin{tikzpicture}[thick, scale=0.8]
    \begin{axis}[
    xmin=-0.5,xmax=5,ymin=0,ymax=5,
    axis lines= middle,
    enlargelimits={abs=1},
    axis line style={latex-stealth}, 
    yticklabels={,,},
    xticklabels={,,},
    ticklabel style={font=\tiny,fill=white},
   tick style={draw=none},
   title={Case $\alpha_j=1$}
        ]

    \draw[thick,gray] (0,2) -- (5,2); 
    \draw[thick,gray] (2.2,0) -- (2.2,5); 
    \draw[thick,gray] (0,4) -- (5,4);  
    \draw[thick,gray] (4.2,0) -- (4.2,5); 
    \draw[thick,dotted] (0,0.2) -- (5,3); 
    \draw[thick,black] (0,0.768) -- (5,3.568); 

\node[below] at (2.2,0) {$j$};
\node[below] at (4.2,0) {$j+1$};
\node[left] at (0,2) {$N+1$};
\node[left] at (0,4) {$N+2$};

\end{axis}
\end{tikzpicture}

\end{tabular}
\caption{In the left diagram, the dotted line corresponds to the inequality $\frac{p}{q}\cdot(j+1)<N+1$, and the solid line corresponds to the equality $\frac{p}{q}\cdot(j+1)=N+1$. In the right diagram, the dotted line corresponds to the inequality $\frac{p}{q}\cdot j<N+1$, and the solid line corresponds to the equality $\frac{p}{q}\cdot j=N+1$.}
\end{figure}
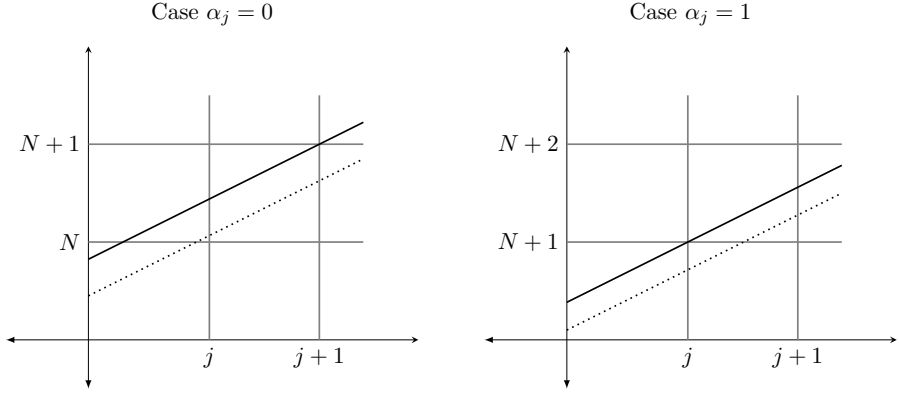

     
     \item $\alpha_j=1$ if and only if 
     \begin{equation}
         \label{Eq2}
         N<\frac{p}{q}\cdot j\leq N+1<\frac{p}{q}\cdot(j+1),
     \end{equation}
     for some $N\geq0$.
In this case $k_{N+j}k_{N+j+1}=01$ and the substitution rule $T$ maps $k_{N+j}k_{N+j+1}\mapsto\alpha_j$. The equality only occurs when $j=q$. See the right diagram in Figure~\ref{fig2.1}.
 \end{itemize} 
Analogously, with convention $10$, $\toz(p/q)=0.\overline{\alpha_1\ldots\alpha_q}$, where 
 \begin{itemize}
     \item $\alpha_j=0$ if and only if 
     \begin{equation}
         \label{Eq3}
         N\leq\frac{p}{q}\cdot j<\frac{p}{q}\cdot(j+1)< N+1,
     \end{equation}
     for some $N\geq0$. In this case $k_{N+j}k_{N+j+1}=00$ and the substitution rule $T$ maps $k_{N+j}\mapsto\alpha_j$. The equality only occcurs when $j=q$.
     \item $\alpha_j=1$ if and only if 
     \begin{equation}
         \label{Eq4}
         N<\frac{p}{q}\cdot j< N+1\leq\frac{p}{q}\cdot(j+1),
     \end{equation}
     for some $N\geq0$.
In this case $k_{N+j}k_{N+j+1}=01$ and the substitution rule $T$ maps $k_{N+j}k_{N+j+1}\mapsto\alpha_j$. The equality only occurs when $j=q-1$.
 \end{itemize} 
\end{remark}

\begin{proposition}\label{prop:ExtAngles}
Given any rational number $p/q\in ]0,1[$, the M-sequences $\tzo(p/q)$ and $\toz(p/q)$ determined by the line $y=\frac{p}{q}x$ are the binary expansions of the characteristic external angles of the root point $c_{p/q}\in \M$, where the principal bulb $H(p/q)$ is attached to the main cardioid of the Mandelbrot set. In particular, $\tzo(p/q)$ and $\toz(p/q)$ have minimal period $q$.
\end{proposition}
\begin{proof}
    To prove this, we compare the binary expansions of the external angles $s_-$ and $s_+$ given in \cite{Dev} with angles $\tzo$ and $\toz$ from our algorithm.
    
    Let $R_{p/q}(\theta)=\theta+p/q$. From \cite{Dev} we know that $$s_{-}=0.\overline{s_1s_2\cdots s_{q-2}01}\quad\text{and}\quad s_{+}=0.\overline{s_1s_2\cdots s_{q-2}10}$$
    are sequences of minimal period $q$, where 
    $$s_j= \left\{ \begin{array}{lcc} 0 & \text{if} & R^{j-1}_{p/q}(p/q)\mod1\in]0,1-p/q[ \\ \\ 1 & \text{if} & R^{j-1}_{p/q}(p/q)\mod1\in]1-p/q,1[\end{array} \right.$$
    
    On the other hand, recall that $\tzo(p/q):=0.\overline{\alpha_1\ldots\alpha_{q-2}01}$ and $\toz(p/q):=0.\overline{\alpha_1\ldots\alpha_{q-2}10}$, see Definition \ref{def:cs}.
    
    Let $1\leq j\leq q-2$, 
    \begin{itemize}
        \item if $s_j=0$, then $R^{j-1}_{p/q}(p/q)\mod1\in]0,1-p/q[$, that is $\frac{jp}{q}=N+\delta$ with $N\in\mathbb{N}$ and $0<\delta<1-p/q$. Therefore $$N<\frac{p}{q}\cdot j<\frac{p}{q}\cdot(j+1)=\frac{p}{q}\cdot j+\frac{p}{q}=N+\delta+\frac{p}{q}<N+1,$$and hence by \eqref{Eq1} in Remark \ref{RemCMS} we have that $\alpha_j=0$.
        \item if $s_j=1$, then $R^{j-1}_{p/q}(p/q)\mod1\in]1-p/q,1[$, that is $\frac{jp}{q}=N+\delta$ with $N\in\mathbb{N}$ and $1-p/q<\delta<1$. Therefore $$N<\frac{p}{q}\cdot j=N+\delta<N+1<N+\delta+\frac{p}{q}=\frac{p}{q}\cdot(j+1),$$and hence by \eqref{Eq2} in Remark \ref{RemCMS} we have that $\alpha_j=1$.
    \end{itemize} 
It follows that $\tzo(p/q)$ and $\toz(p/q)$ have minimal period $q$.

\end{proof}

\begin{lemma} \label{lemma:L0}
Let $\frac{a}{b}$ and $\frac{c}{d}$ be Farey neighbors such that $0<\frac{a}{b}<\frac{c}{d}\leq1$ and let $\frac{p}{q}$ be their mediant. Consider the cutting sequences $\kappa_{01}(a/b)=\overline{k_1\cdots k_{a+b}}$ and $\kappa_{01}(c/d)=\overline{t_1\cdots t_{c+d}}$. Then
\[\kappa_{01}(p/q)=\overline{t_1\cdots t_{c+d}k_1\cdots k_{a+b}}.\]
Analogously, if  $0\leq \frac{a}{b}<\frac{c}{d}<1$, $\frac pq$ is their mediant, $\kappa_{10}(a/b)=\overline{k_1\cdots k_{a+b}}$ and $\kappa_{10}(c/d)=\overline{t_1\cdots t_{c+d}}$, then
\[\kappa_{10}(p/q)=\overline{k_1\cdots k_{a+b} t_1\cdots t_{c+d}}.\]
\end{lemma}
\begin{proof}
It will be sufficient to prove the result for the 01-convention as the proof is analogous for the 10-convention.
By hypothesis, one has
 $\frac{p}{q}-\frac{a}{b}=\frac{1}{bq}$ and $\frac{c}{d}-\frac{p}{q}=\frac{1}{dq}$. Let $\kappa_{01}(\frac{p}{q})=\overline{g_1\cdots g_{p+q}}$.
 
We first prove that $g_1\cdots g_{c+d-2}=t_1\cdots t_{c+d-2}$. If this is not the case, there exists $1\leq j\leq c+d-2$ such that $g_1\cdots g_{j-1}= t_1\cdots t_{j-1}$ and $g_j\neq t_j$. Since $\frac{c}{d}>\frac{p}{q}$ then necessarily $g_j=0$ and $t_j=1$, and since $\frac{c}{d}\leq1$ then $t_{j+1}=0$.\\
Assume that the word $g_1\ldots g_j$ contains exactly $s$ entries equal to 1 and, therefore, $j-s$ zeros. This means that $\frac{p}{q}(j-s)=s+\frac{r}{q}$ for some $r\in\{1,\ldots,q-1\}$. This implies that $t_1\ldots t_{j+1}$ contains exactly $s+1$ entries equal to 1 and $j-s$ zeros, so then $\frac{c}{d}(j-s)=s+1+\frac{m}{d}$ for some $m\in\{1,\ldots,d-1\}$. Note that 
\[\frac{1}{q}+\frac{1}{d}\leq s+1+\frac{m}{d}-(s+\frac{r}{q})=\frac{c}{d}(j-s)-\frac{p}{q}(j-s)=\frac{j-s}{dq} <\frac{d}{dq}=\frac{1}{q},\]
which is a contradiction. Therefore, $g_1\cdots g_{c+d-2}=t_1\cdots t_{c+d-2}$. The 01-convention implies that $t_{c+d-1} t_{c+d}=01$ and since
\[\frac{p}{q} d=\left(\frac{c}{d}-\frac{1}{dq}\right) d=c-\frac{1}{q}<c<c+\frac{p-1}{q}=c-\frac{1}{q}+\frac{p}{q}=\frac{p}{q}(d+1),\]
we conclude that $g_{c+d-1}g_{c+d}=01$, therefore, $g_1\cdots g_{c+d}=t_1\cdots t_{c+d}$.

Now we prove that $g_{c+d+1}\cdots g_{p+q-2}=k_1\cdots k_{a+b-2}$. Since $\frac{p}{q} d=c-\frac{1}{q}$, then we prove the equality of the words by comparing the line $y=\frac{a}{b}x$ with the line $y=\frac{p}{q}x-\frac{1}{q}$ over the interval $[1,b]$. Assume that there exists $1\leq j\leq a+b-2$ so that $g_{c+d+1}\cdots g_{c+d+j-1}=k_1\cdots k_{j-1}$ and $g_{c+d+j}\neq k_j$. Since $\frac{p}{q}x-\frac{1}{q}<\frac{a}{b}x$ for all $x<b$, then $k_j=1$ and $g_{c+d+j}=0$, and since $\frac{a}{b}<1$ then $k_{j+1}=0$.

Suppose the word $g_{c+d+1}\ldots g_{c+d+j}$ contains exactly $s$ entries equal to 1 and, therefore, $j-s$ zeros. This means that $\frac{p}{q}(j-s)-\frac{1}{q}=s+\frac{r}{q}$, for some $r\in\{1,\ldots,q-1\}$. This also implies that $k_1 \ldots k_{j+1}$ contains exactly $s+1$ entries equal to 1 and $j-s$ zeros, so then $\frac{a}{b}(j-s)=s+1+\frac{m}{b}$, for some $m\in\{1,\ldots,b-1\}$. Note that 
\[\frac{1}{q}+\frac{1}{b}\leq s+1+\frac{m}{b}-(s+\frac{r}{q})=\left(\frac{a}{b}-\frac{p}{q}\right)(j-s)+\frac{1}{q}=\frac{1}{q}-\frac{j-s}{bq},
\]
which is a contradiction. Therefore, $g_{c+d+1}\cdots g_{p+q-2}= k_1\cdots k_{a+b-2}$. Finally, the 01-convention implies that $k_{a+b-1}k_{a+b}=g_{p+q-1}g_{p+q}=01$.
\end{proof}

We can now derive the next two fundamental results that explain how to compute M-sequences of a given fraction from the respective M-sequences of its Farey neighbors. The first result follows directly from Lemma~\ref{lemma:L0} after applying the substitution rule to each cutting sequence. Compare with the results in \cite[Section 7]{Dev}.

\begin{proposition}\label{prop:P1}
Let $\frac{a}{b}$ and $\frac{c}{d}$ be Farey neighbors such that $0<a/b<c/d\leq 1$ and let $p/q$ be their mediant. If the M-sequences of $a/b$ and $c/d$ are given respectively by $\tzo(a/b)=0.\overline{W^{01}_{a/b}}$ and $\tzo(c/d)=0.\overline{W^{01}_{c/d}}$, then
\begin{equation}\label{eq:concatF}
\tzo(p/q)=0.\overline{W^{01}_{c/d} W^{01}_{a/b}}.
\end{equation}
Analogously, if $0\leq a/b<c/d< 1$ the M-sequence of $p/q$ in the $10$-convention is
\begin{equation}\label{eq:concatE}
\toz(p/q)=0.\overline{W_{a/b}^{10} W_{c/d}^{10}}.
\end{equation}
\end{proposition}

\begin{corollary}\label{cor:Cor1}
If $p/q$ is any rational number such that $0<\frac{a}{b}<\frac{p}{q}<\frac{c}{d}\leq 1$, then
\[\tzo(p/q)=.\overline{W_{c/d}^{01}W_{\zeta_1}^{01}\cdots W_{\zeta_k}^{01} W_{a/b}^{01}},\]
for some $k\geq0$ and $\zeta_i\in\{c/d,a/b\}$ with $i=1,\ldots,k$. Similarly, if $0\leq \frac{a}{b}<\frac{p}{q}<\frac{c}{d}< 1$ one has $\toz(p/q)=.\overline{W_{a/b}^{10}W_{\zeta_1}^{10}\cdots W_{\zeta_k}^{10} W_{c/d}^{10}}$, with the same conditions over $k,\zeta_i$ and $i$.
\end{corollary}
\begin{proof}
It follows inductively from Proposition \ref{prop:P1}.
\end{proof}

\subsection{M-sequences of broken lines for the 01-convention}
From now on we will work only with M-sequences under the $01$-convention and leave the discussion of analogous results for the $10$-convention in Appendix~\ref{App:A}. We will omit the reference to the $01$-convention from words and blocks of words whenever possible.

Let $\frac PQ, \frac ab, \frac AB, \frac ST$ satisfy the 01-hypothesis in \eqref{eq:MainHyp}. In this section we study  M-sequences of broken lines as in Definition~\ref{def:BL} with an upper bound on $a/b$. For further reference, we also consider the rational numbers
\begin{equation}\label{eq:PmQm}
\frac{P}{Q}\leq \frac{P_m}{Q_m}<\frac{S_n}{T_n}\qquad\text{where}\qquad \frac{P_m}{Q_m}:=\frac{P+mS_n}{Q+mT_n}\quad\text{for}~m\geq 0.
\end{equation}

\begin{lemma}[Periodic Sturmian angles]\label{lem:persturm}
Fix $n\geq 1$ and consider $\frac PQ, \frac ST$ as in \eqref{eq:MainHyp}. If $\frac ab$ is selected so that $\frac PQ<\frac ab< \frac{S_n}{T_n}$, then $\tzo(P/Q, a/b, n)$ is the binary expansion of a periodic Sturmian angle of period $b$.
\end{lemma}

\begin{proof}
The given hypotheses and Corollary \ref{cor:Cor1} show that 
$W_{a/b}=W_{S/T}W_{\zeta_1}\cdots W_{\zeta_k}W_{P/Q}^n$,
for some $k\geq0$ and $\zeta_i\in\{P/Q,S/T\}$ for $i=1,\ldots,k$. As a consequence of Proposition~\ref{prop:ExtAngles}, $W_{a/b}$ has minimal period $b$. By the geometry of a broken line at the $n^{\text{th}}$ hinge point in Lemma~\ref{lem:easybl} we conclude 
\begin{equation}\label{eq:perBL}
\tzo\left (\frac PQ, \frac a b,n\right)= 0.W_{P/Q}^n\overline{W_{S/T}W_{\zeta_1}\cdots W_{\zeta_k}W_{P/Q}^n}=0.\overline{W_{P/Q}^nW_{S/T}W_{\zeta_1}\cdots W_{\zeta_k}}.
\end{equation}
\end{proof}

\begin{proposition}\label{prop:SnTn}
Let $\frac{P}{Q}\leq\frac{P_m}{Q_m}<\frac{S_n}{T_n}$ for some $n\geq 1$ and $m\geq 1$. Then
\[\tzo\left(\frac P Q,\frac{P_m}{Q_m}, n\right)=0.\overline{W_{P/Q}^n(W_{S/T}W_{P/Q}^{n-1})^{m-1}W_{S/T}}.\]
\end{proposition}
\begin{proof}
Proposition \ref{prop:P1} applied to $\frac{P_m}{Q_m}$ and Corollary~\ref{cor:Cor1} applied to $\frac{S_n}{T_n}$ shows that
\[\tzo(P_m/Q_m)=0.\overline{W_{S_n/T_n}^m W_{P/Q}}=0.\overline{(W_{S/T}W_{P/Q}^{n-1})^m W_{P/Q}}.\]
Combining this with Lemma~\ref{lem:persturm}, the M-sequence of the broken line $BL(\frac PQ,\frac{a}{b},n)$ becomes
\begin{eqnarray*}
0.W_{P/Q}^n\overline{(W_{S/T}W_{P/Q}^{n-1})^m W_{P/Q}}&=&0.W_{P/Q}^n\overline{(W_{S/T}W_{P/Q}^{n-1})^{m-1}W_{S/T}W_{P/Q}^n}\\
&=&0.\overline{W_{P/Q}^n(W_{S/T}W_{P/Q}^{n-1})^{m-1}W_{S/T}}.
\end{eqnarray*}
\end{proof}

Next, we obtain a relationship between the M-sequences of broken lines of the Farey neighbors and their Farey descendants.

\begin{proposition}\label{prop:concat1}
Let $a/b$ and $c/d$ be Farey neighbors and $f/g$ their mediant, so that $\frac{P}{Q}\leq\frac{a}{b}<\frac{f}{g}<\frac{c}{d}<\frac{S_n}{T_n}$. If the M-sequences of the broken lines $BL(\frac PQ,\frac a b,n)$ and $BL(\frac PQ,\frac c d,n)$ are expressed digitwise as $0.\overline{\alpha_1\cdots\alpha_b}$ and $0.\overline{\beta_1\cdots\beta_d}$ respectively, then the M-sequence of $BL(\frac PQ,\frac f g,n)$ is $0.\overline{\beta_1\cdots\beta_d\alpha_1\cdots\alpha_b}$.
\end{proposition}

\begin{proof}
Reversing the deduction of (\ref{eq:perBL}) in the proof of Lemma~\ref{lem:persturm} we have that
\[\alpha_1\cdots\alpha_b=W_{P/Q}^n\alpha_{nQ+1}\cdots\alpha_b\hspace{1cm}\text{ and }\hspace{1cm}\beta_1\cdots\beta_d=W_{P/Q}^n\beta_{nQ+1}\cdots\beta_d,\]
and hence, the M-sequences associated with $\frac{a}{b}$ and $\frac{c}{d}$ are 
\[\tzo(a/b)=0.\overline{\alpha_{nQ+1}\cdots\alpha_b W_{P/Q}^n}\hspace{1cm}\text{ and }\hspace{1cm}\tzo(c/d)=0.\overline{\beta_{nQ+1}\cdots\beta_d W_{P/Q}^n}.\] By Proposition \ref{prop:P1}, the M-sequence of $\frac{f}{g}$ is
\[\tzo(f/g)=0.\overline{\beta_{nQ+1}\cdots\beta_d W_{P/Q}^n\alpha_{nQ+1}\cdots\alpha_b W_{P/Q}^n},\]
and therefore, Lemma~\ref{lem:easybl} states that the M-sequence $\tzo(\frac PQ,\frac{f}{g},n)$ becomes
\begin{align*}
0.W_{P/Q}^n\overline{\beta_{nQ+1}\cdots\beta_d W_{P/Q}^n\alpha_{nQ+1}\cdots\alpha_b W_{P/Q}^n}&=0.\overline{W_{P/Q}^n\beta_{nQ+1}\cdots\beta_d W_{P/Q}^n\alpha_{nQ+1}\cdots\alpha_b}\\ 
    &=0.\overline{\beta_1\cdots\beta_d\alpha_1\cdots\alpha_b}.
\end{align*}
\end{proof}

\begin{remark}\label{rem:blocknotation}
On the basis of the expression of the M-sequence in Proposition~\ref{prop:SnTn}, we introduce the following block notation: let $0\leq\frac{A}{B}<\frac{S}{T}\leq1$ be Farey neighbors, and let $\frac{P}{Q}$ be their mediant. For $n\geq 1$ and $m\geq 1$ we define the blocks
\begin{equation}\label{def:D1}
B_{n,m}:=W_{P/Q}^n(W_{S/T}W_{P/Q}^{n-1})^{m-1}W_{S/T},\qquad\text{and}\qquad
B_{n,0}:=W_{P/Q}.
\end{equation}
\end{remark}

Therefore, the conclusion in Proposition~\ref{prop:SnTn} is written as $\tzo\left(\frac PQ,\frac{P_m}{Q_m},n\right)=0.\overline{B_{n,m}}$. The next result is easily proved via induction.
\begin{lemma}\label{lem:R1}
Consider any rational number $\frac ab$ so that $\frac{P_m}{Q_m}<\frac ab < \frac{P_{m+1}}{Q_{m+1}}$ for some $m\geq0$. Then, there exist $k\geq2$ and $m_i\in\{m,m+1\}$ for $i=1,\ldots,k$ so that
\[\tzo \left(\frac PQ,\frac ab,n\right)=0.\overline{B_{n,m_1}B_{n,m_2}\cdots B_{n,m_{k-1}}B_{n,m_k}}=0.\overline{B_{n,m+1}B_{n,m_2}\cdots B_{n,m_{k-1}}B_{n,m}}.\]
\end{lemma}

\subsection{Comparing M-sequences}
This section presents two lemmas, which are central to the proof of Theorem~\ref{thm:B}. The main idea behind each result is to compare two given M-sequences: Given any two Farey neighbors, the first lemma shows how to find the first digit where their M-sequences differ from each other. Then, Lemma~\ref{lem:Lemy0} shows how to compare the M-sequence of a broken line with any wordwise shift of the same sequence. A blockwise statement of Lemma~\ref{lem:Lemy0} is given in Corollary~\ref{cor:Lem1}.

\noindent
\emph{Notation}: Consider a finite set of symbols $\alpha_1,\ldots,\alpha_k\in \{0,1\}$ and denote by $\mathcal{D}$ the collection of dyadic numbers.
We define the collection of binary expansions
\[[0.\alpha_1\cdots \alpha_k]=\{x\in [0,1]\setminus \mathcal{D}~:~ x=0.\alpha_1\cdots \alpha_k x_{k+1}\cdots\}.\] 
Given two collections  $[0.\alpha_1\cdots \alpha_k]$ and $[0.\beta_1\cdots \beta_{k'}]$, let $l=\min\{k,k'\}$, then we write
\[
[0.\alpha_1\cdots \alpha_k]<[0.\beta_1\cdots \beta_{k'}]\qquad\text{if and only if}\qquad 0.\alpha_1\cdots \alpha_l\bar{0}<0.\beta_1\cdots \beta_{l}\bar{0}.
\]
In this case, if $x\in[0.\alpha_1\alpha_2\cdots \alpha_k]$ and $y\in[0.\beta_1 \beta_2\cdots \beta_{k'}]$ then $x<y$.

\begin{lemma}\label{lem:L1}
Let $0<\frac{a}{b}<\frac{c}{d}\leq1$ be Farey neighbors and consider the M-sequences
\[\tzo(a/b)=0.\overline{\alpha_1\cdots\alpha_b}=0.\alpha_1\cdots\alpha_b\alpha_{b+1}\cdots\alpha_{2b}\cdots,\]
and 
\[\tzo(c/d)=0.\overline{\beta_1\cdots\beta_d}=0.\beta_1\cdots\beta_d\beta_{d+1}\cdots\beta_{2d}\cdots.\]
Then, there exists an integer $1\leq r\leq b$ such that $\alpha_i=\beta_i$ for all $1\leq i<r$ and $\alpha_r<\beta_r$.
\end{lemma}

\begin{proof}
We prove this lemma inductively. First, assume that $\frac{a}{b}=\frac{1}{n+1}$ and $\frac{c}{d}=\frac{1}{n}$ for some $n\in\mathbb{N}$. As shown in (\ref{eq:An}), the M-sequences for these fractions are
\[\tzo(1/(n+1))=\overline{0^n1}\qquad \text{and} \qquad \tzo(1/n)=\overline{0^{n-1}1}.\]
Therefore, the lemma holds true for $r=n$. As the induction step, assume that for the given Farey neighbors $0<\frac{a}{b}<\frac{c}{d}\leq1$ and the integer $n\geq 0$ that satisfies
\[\frac{1}{n+1}\leq\frac{a}{b}<\frac{c}{d}\leq\frac{1}{n},\] 
the conclusion of the lemma holds true for $\frac{a}{b}$ and $\frac{c}{d}$. We prove that the lemma also holds true  for the mediant $\frac{p}{q}=\frac{a+c}{b+d}$ and any of its Farey parents. From Proposition \ref{prop:P1}, the M-sequence of $\frac{p}{q}$ is
\[\tzo(p/q)=0.\overline{\gamma_1\cdots\gamma_q}=0.\overline{\beta_1\cdots\beta_d\alpha_1\cdots\alpha_b}.\]
By the induction step, there exists $1\leq r\leq b$ such that $\alpha_i=\beta_i$ for all $i<r$ and $\alpha_r<\beta_r$. 

\emph{Case 1.} Consider the Farey neighbors $\frac{a}{b}<\frac{p}{q}$. Write $r=kd+s$ for $k\geq0$ and $0\leq s<d$. Then $\alpha_r=0$, $\beta_r=\beta_s=1$ and therefore
\[0.\overline{\alpha_1\cdots\alpha_b}=0.\overline{\underbrace{\beta_1\cdots\beta_d\cdots\beta_1\cdots\beta_d}_{kd\text{ digits}}{\beta_1}\cdots\beta_{s-1}\underbrace{0}_{\alpha_r}\alpha_{r+1}\cdots\alpha_b}.\]

On the other hand,
\[0.\overline{\gamma_1\cdots\gamma_q}=0.\overline{\underbrace{\beta_1\cdots\beta_d \cdots\beta_1\cdots \beta_d}_{kd\text{ digits}}\beta_1\cdots\beta_{s-1}\underbrace{1}_{\gamma_r} \beta_{s+1}\cdots \beta_d \beta_1\cdots\beta_{s-1}0\cdots\alpha_b}.\]
Therefore, $\alpha_i=\gamma_i$ for all $i<r$, and $\alpha_r=0<1=\beta_s=\gamma_r$.

\emph{Case 2.} Consider the Farey neighbors $\frac{p}{q}<\frac{c}{d}$. Let $r'=r+d$, then $1+d\leq r'\leq b+d=q$. If $i\leq d$ then trivially $\gamma_i=\beta_i$, whereas, if $d+1\leq i<r'$ then by hypothesis
\[\gamma_i=\alpha_{i-d}=\beta_{i-d}=\beta_i,\]
and hence $\gamma_{r'}=\alpha_{r'-d}=\alpha_r=0<1=\beta_r=\beta_{r'}$.
\end{proof}

Based on the previous result, we now consider the M-sequence of a Farey descendant and establish a comparison between the sequence and any of its wordwise shifts.

\begin{lemma}\label{lem:Lemy0}
Let $0<\frac{a}{b}<\frac{c}{d}\leq1$ be Farey neighbors with M-sequences $\tzo(a/b)=0.\overline{W_{a/b}}$ and $\tzo(c/d)=0.\overline{W_{c/d}}$. Given any $\frac pq$ so that $\frac{a}{b}<\frac{p}{q}<\frac{c}{d}$, then 
\begin{enumerate}
\item[(1)] There exists an integer $k=k(p/q)\geq 2$ such that $\tzo(p/q)=0.\overline{W_{\zeta_1}\cdots W_{\zeta_k}}$ with $\zeta_j\in\{a/b,c/d\}$ for all $j=1,\ldots, k$. In particular $W_{\zeta_1}=W_{c/d}$ and $W_{\zeta_k}=W_{a/b}$.
\item[(2)] For each $1<j\leq k$ there exists an integer $0\leq r=r(j,k)\leq k-j$ such that $\zeta_{1+i}=\zeta_{j+i}$ for all $0\leq i<r$ and $\zeta_{1+r}>\zeta_{j+r}$.
\item[(3)] 
$\left[0.W_{\zeta_j}\cdots W_{\zeta_k}\right]<\left[0.W_{\zeta_1}\cdots W_{\zeta_k}\right]$ for all $1<j\leq k$.
\end{enumerate}
\end{lemma}

\begin{proof}
The statement in (1) follows directly from Corollary~\ref{cor:Cor1}. To see (2), we divide the proof into four cases.

\emph{Case 1.} If $\frac{p}{q}$ is the mediant of $\frac{a}{b}$ and $\frac{c}{d}$, then $\tzo(p/q)=0.\overline{W_{c/d}W_{a/b}}$, so statement (2) is trivially true with $j=2$ and $r=0$.

\emph{Case 2.} Assume $\frac{p}{q}$ is the mediant of $\frac{a}{b}$ and $\frac{x}{y}$ where $\frac{a}{b}<\frac xy<\frac{c}{d}$. Statement (1) implies that $\tzo(x/y)=0.\overline{W_{c/d}W_{\zeta_2}\cdots W_{\zeta_{k-1}}W_{a/b}}$ for some $2\leq k=k(x/y)$ and $\zeta_i\in\{a/b,c/d\}$ for all $i=1,\ldots,k$. Assume that statement (2) holds true for $\frac{x}{y}$, so for each $1<j\leq k$ there is an integer $0\leq r\leq k-j$ with the required properties. Since $r$ depends on $j$ and $k$, it also depends on $\frac xy$. Therefore, we write $r=r(x/y)$. Proposition~\ref{prop:P1} shows that $\tzo(p/q)$ is given by the concatenation of $k+1$ words in $W_{a/b}$ and $W_{c/d}$, that is
\[\tzo(p/q)=0.\overline{W_{c/d}W_{\zeta_2}\cdots W_{\zeta_{k-1}}W_{a/b}W_{a/b}},\qquad \zeta_i\in \{a/b,c/d\}.\]
Let $1<j\leq k+1$. 
\begin{itemize}
\item If $j\leq k$ then set $r(p/q)=r=r(x/y)$. Then $0\leq r(p/q)\leq k-j<k+1-j$,  $\zeta_{1+i}=\zeta_{j+i}$ for all $i<r$ and $\zeta_{1+r}>\zeta_{j+r}$.
\item If $j=k+1$, then set $r=0$. Then $\zeta_{1+r}=\zeta_1=\frac cd>\frac ab=\zeta_{k+1}=\zeta_{j+r}$.
\end{itemize}

\emph{Case 3.} Assume that $\frac{p}{q}$ is the mediant of $\frac{c}{d}$ and $\frac{x}{y}$ where $\frac{a}{b}<\frac xy<\frac{c}{d}$. The same arguments as in Case 2 apply to $\frac xy$ to obtain $\tzo(x/y)=0.\overline{W_{c/d}W_{\zeta_2}\cdots W_{\zeta_{k-1}}W_{a/b}}$, for some $2\leq k=k(x/y)$ and $\zeta_i\in\{a/b,c/d\},i=1,\ldots,k$. Assume that statement (2) holds true for $\frac{x}{y}$, so for each $1<j\leq k$ there is an integer $0\leq r=r(x/y)\leq k-j$ with the required properties. Once again $\tzo(p/q)=0.\overline{W_{c/d}W_{x/y}}$ which becomes the concatenation of $k+1$ words in $W_{a/b}$ and $W_{c/d}$. For simplicity, let
\[\tzo(p/q)=0.\overline{ W_{c/d} W_{c/d}W_{\zeta_2}\cdots W_{\zeta_{k-1}}W_{a/b}}
=0.\overline{W_{\chi_1}\cdots W_{\chi_{k+1}}},\]
where $\zeta_i, \chi_i\in \{a/b,c/d\}$, $W_{\chi_1}=W_{c/d}$ and $\zeta_i=\chi_{i+1}$ for all $i=1,\ldots,k$. Let $1<j\leq k+1$. 
\begin{itemize}
\item If $j=2$ then $\chi_j=c/d$. Set $r=r(p/q)$ equal to the smallest integer for which $\chi_{j+r}=a/b$. This choice of $r$ satisfies the conclusion of statement (2) since 
\begin{enumerate}
\item $1\leq r\leq k+1-j$,
\item $\chi_{1+i}=\chi_{j+i}=c/d$ for all $0\leq i<r$ and
\item $c/d=\zeta_{1+r}>\zeta_{j+r}=a/b$.
\end{enumerate}

\item Now assume $j>2$. Since statement (2) holds for $\frac{x}{y}$, then for $j_1:=j-1$, there exists $0\leq r_1\leq k-j_1$ such that $\zeta_{1+i}=\zeta_{j_1+i}$ for all $i<r_1$ and $\zeta_{1+r_1}>\zeta_{j_1+r_1}$. This implies that $\chi_{2+i}=\chi_{j+i}$ for all $i<r_1$ and $\chi_{2+r_1}>\chi_{j+r_1}$.

Let $r_2$ be the smallest integer such that $\chi_{j+r_2}=a/b$ and let $r=\min\{r_1,r_2\}$. Note that if $0\leq i<r$ then $\chi_{2+i}=\chi_{j+i}=c/d$ and therefore $\chi_{1+i}=c/d=\chi_{j+i}$ for all $i<r$ and
\[\chi_{1+r}=\chi_{2+r-1}=\chi_{j+r-1}=\frac cd>\frac ab=\chi_{j+r}.\]
\end{itemize}

\emph{Case 4.} Assume that $\frac{p}{q}$ is the mediant of Farey neighbors $\frac{x}{y}$ and $\frac{s}{t}$, with $\frac{a}{b}<\frac{x}{y}<\frac{p}{q}<\frac{s}{t}<\frac{c}{d}$. Then for some $k,l\geq 2$, $\nu_i, \chi_i \in \{c/d,a/b\}$ we have
\begin{equation}\label{eq:dummy2}
\tzo(x/y)=0.\overline{W_{c/d}W_{\nu_2}\cdots W_{\nu_{k-1}}W_{a/b}}\quad\text{and}\quad \tzo(s/t)=0.\overline{W_{c/d}W_{\chi_2}\cdots W_{\chi_{l-1}}W_{a/b}}.
\end{equation}
Assume statement (2) holds for both $\frac{x}{y}$ and $\frac{s}{t}$. The M-sequence $\tzo(p/q)$ can be expressed as
\begin{equation}\label{eq:dummy1}
0.\overline{W_{c/d}W_{\chi_2}\cdots W_{\chi_{l-1}}W_{a/b}W_{c/d}W_{\nu_2}\cdots W_{\nu_{k-1}}W_{a/b}}=0.\overline{W_{\zeta_1}\cdots W_{\zeta_{k+l}}},
\end{equation}
where $\zeta_i=\chi_i$ for $i=1,\ldots, l$, and $\zeta_{l+i}=\nu_i$ for $i=1,\ldots,k$.
Let $1<j\leq k+l$.
\begin{itemize}
\item If $1<j\leq l$ and since statement (2) holds for $\frac{s}{t}$, there exists $0\leq r=r(s/t)\leq l-j<k+l-j$ such that $\zeta_{1+i}=\chi_{1+i}=\chi_{j+i}=\zeta_{j+i}$ for all $i<r$ and $\zeta_{1+r}=\chi_{1+r}>\chi_{j+r}=\zeta_{j+r}$. Therefore $r(p/q)=r(s/t)$.

\item Assume $j=l+1$ so that $W_{\zeta_j}=W_{c/d}$. Let $r=r(p/q)$ be the first integer such that $\zeta_r\neq\nu_r$. Note that $r>1$ and by Lemma \ref{lem:L1}, $r\leq k$, hence $\zeta_i=\nu_i$ for all $i<r$ and again, by Lemma \ref{lem:L1}, $\zeta_r>\nu_r$.
\item Finally, assume $l+1<j\leq k+l$ and set $j_1=j-l>1$. From (\ref{eq:dummy2}),  (\ref{eq:dummy1}) and since $\zeta_{l+i}=\nu_i$ for $i=1,\ldots,k$, we can write
\[\tzo(x/y)= 0.\overline{W_{\zeta_{l+1}}\cdots W_{\zeta_{k+l}}}=0.W_{\zeta_{l+1}}W_{\zeta_{l+2}}\cdots W_{\zeta_{k+l}}W_{\zeta_{l+1}}\cdots.\]
On the one hand, statement (2) holds for $x/y$ so we can find $0\leq r_1=r_1(x/y)\leq k-j_1=k+l-j$ such that $\zeta_{l+1+i}=\nu_{1+i}=\nu_{j_1+i}=\zeta_{j+i}$ for all $i<r_1$ and $\zeta_{l+1+r_1}=\nu_{1+r_1}>\nu_{j_1+r_1}=\zeta_{j+r_1}$.However, since $\frac xy<\frac pq$ then by comparing their M-sequences we can find the smallest integer $0\leq r_2\leq k$ for which $\zeta_{1+r_2}>\zeta_{l+1+r_2}=\nu_{1+r_2}$. Let $r=\min\{r_1,r_2\}$. Note that $\zeta_{1+i}=\zeta_{l+1+i}=\zeta_{j+i}$ for all $i<r$ and $\zeta_{1+r}>\zeta_{j+r}$.

\end{itemize}

We now prove statement (3). First note that for each $l\geq 0$ the fractions $\frac{a+lc}{b+ld}$ and $\frac{a+(l+1)c}{b+(l+1)d}$ are Farey neighbors with $M$-sequences $0.\overline{(W_{c/d})^lW_{a/b}}$ and $0.\overline{(W_{c/d})^{l+1}W_{a/b}}$ respectively. Lemma \ref{lem:L1} implies that
\[\left[0.W_{a/b}\right]<\left[0.W_{c/d}W_{a/b}\right]<\cdots<\left[0.(W_{c/d})^{l} W_{a/b}\right].\]

Let $\tzo(p/q)=0.\overline{W_{\zeta_1}\ldots W_{\zeta_{k}}}=0.\overline{W_{c/d}W_{\zeta_2}\ldots W_{\zeta_{k-1}}W_{a/b}}$ with $\zeta_i\in \{a/b, c/d\}$ and some $k\geq 2$. For any given $1<j\leq k$, we want to compare $0.W_{\zeta_j}\ldots W_{\zeta_k}$ with $0.W_{\zeta_1}\ldots W_{\zeta_k}$. For such $j$, Statement (2) implies the existence of $0\leq r\leq k-j$ such that $\zeta_{1+i}=\zeta_{j+i}$ for all $0\leq i<r$ and $\zeta_{1+r}=\frac cd>\frac ab=\zeta_{j+r}$. Therefore we can write
\[W_{\zeta_j}\cdots W_{\zeta_{j+r-1}}W_{\zeta_{j+r}}\cdots W_{\zeta_k}= W_{\zeta_1}\cdots W_{\zeta_{r}}W_{a/b}\cdots W_{\zeta_k}.\]
Moreover, since $W_{\zeta_{1+r}}=W_{c/d}$, let $l\geq1$ be the smallest integer so that $W_{\zeta_{1+r+l}}=W_{a/b}$ and
\[W_{\zeta_{1+r}}\cdots W_{\zeta_{r+l}}W_{\zeta_{1+r+l}}=(W_{c/d})^l W_{a/b}.\]
Combining all the above expressions, we obtain
\begin{align*}
0.W_{\zeta_j}\cdots W_{\zeta_k}\bar{0}  &\in\left[0.W_{\zeta_1}\cdots W_{\zeta_{r}}W_{a/b}\cdots W_{\zeta_k}\right]\\
     &<\left[0.W_{\zeta_1}\cdots W_{\zeta_{r}}(W_{c/d})^{l} W_{a/b}\cdots W_{\zeta_k}\right]\ni0.W_{\zeta_1}\cdots W_{\zeta_k}\bar{0},
\end{align*}
therefore $[0.W_{\zeta_j}\cdots W_{\zeta_k}]<[0.W_{\zeta_1}\cdots W_{\zeta_k}]$.

\end{proof}

We can obtain an analogous result of Lemma~\ref{lem:Lemy0} for broken lines and their block representation $B_{n,m}$ defined in (\ref{def:D1}).

\begin{corollary} \label{cor:Lem1}
Fix $n\geq 1$ and let $\frac{P}{Q}, \frac{a}{b},\frac{S_n}{T_n}$ be rational numbers satisfying (\ref{eq:MainHyp}). Assume the M-sequence of the broken line $BL(\frac PQ,\frac ab,n)$ is written as
\[\tzo\left(\frac PQ, \frac ab,n\right) = 0.\overline{B_{n,m_1}\cdots B_{n,m_k}}\]
for some $k>1$ and integers $m_1,\ldots,m_k\geq 0$.
\begin{enumerate}
\item[(1)] For each $1<j\leq k$, there exists an integer $0\leq r\leq k-j$ such that $m_{1+i}=m_{j+i}$ for all $i<r$ and $m_{1+r}>m_{j+r}$.
\item[(2)]  $\left[0.B_{n,m_j}\cdots B_{n,m_k}W_{P/Q}^n\right]<\left[0.B_{n,m_1}\cdots B_{n,m_k}W_{P/Q}^n\right]$ for all $1<j\leq k$.
\end{enumerate}
\end{corollary}

\begin{proof}
Lemma~\ref{lem:R1} shows that the assumption on  the M-sequence of the broken line is equivalent to the existence of an integer $m\geq 0$ so that  $\frac{P_m}{Q_m}<\frac{a}{b}<\frac{P_{m+1}}{Q_{m+1}}$, and therefore,
\[\tzo\left(\frac PQ, \frac ab,n\right) = 0.\overline{B_{n,m+1}B_{n,m_2}\cdots B_{n,m_{k-1}}B_{n,m}},\]
for some $k\geq2$ and with $m_i\in\{m,m+1\}$ for all $2\leq i\leq k-1$. The geometry of the broken line also shows that
\[0.\overline{B_{n,m+1}B_{n,m_2}\cdots B_{n,m_{k-1}}B_{n,m}}=0.W_{P/Q}^n\overline{W_{\zeta_1}\ldots W_{\zeta_k}}\]
where $\tzo(a/b)=0.\overline{W_{\zeta_1}\ldots W_{\zeta_k}}$ and $\zeta_i\in \{\frac{P_m}{Q_m},\frac{P_{m+1}}{Q_{m+1}}\}$. Corollary~\ref{cor:Cor1} shows that $\zeta_1=\frac{P_{m+1}}{Q_{m+1}}$ and $\zeta_k=\frac{P_m}{Q_m}$. 

Lemma~\ref{lem:Lemy0} when applied to $\tzo(a/b)$ shows that for any given $1<j\leq k$, there exists $0\leq r\leq k-j$ so that $\zeta_{1+i}=\zeta_{j+i}$ for all $0\leq i<r$ and $\zeta_{1+r}=\frac{P_{m+1}}{Q_{m+1}}>\frac{P_m}{Q_m}=\zeta_{j+r}$. From the definition of $B_{n,m}$ blocks in (\ref{def:D1}), it is easy to verify that $W_{P/Q}^n W_{\zeta_j}=W_{P/Q}^n W_{P_{m_j}/Q_{m_j}}=B_{n,m_j}W_{P/Q}^n$ for all $j$. Then, the conclusion from Lemma~\ref{lem:Lemy0} translates into $m_{1+i}=m_{j+i}$ for all $0\leq i<r$ and
\begin{align*}
B_{n,m_{1+r}}W_{P/Q}^n=W_{P/Q}^n W_{\zeta_{1+r}}=W_{P/Q}^nW_{P_{m+1}/Q_{m+1}}&\qquad\text{hence $m_{1+r}=m+1$} \\
B_{n,m_{j+r}}W_{P/Q}^n=W_{P/Q}^nW_{\zeta_{j+r}}=W_{P/Q}^nW_{P_m/Q_m} &\qquad\text{hence $m_{j+r}=m$.}
\end{align*}
Therefore $m_{1+r}>m_{j+r}$. To see statement (2), consider again $\tzo(a/b)=0.\overline{W_{\zeta_1}\cdots W_{\zeta_k}}$. Part (3) in Lemma~\ref{lem:Lemy0} shows that for all $1<j\leq k$,
\begin{align*}
0.B_{n,m_j}\cdots B_{n,m_k}W_{P/Q}^n\cdots&\in [0.W_{P/Q}^n W_{\zeta_j}\cdots W_{\zeta_k}]\\
&<[0.W_{P/Q}^n W_{\zeta_1}\cdots W_{\zeta_k}]\ni 0.B_{n,m_1}\cdots B_{n,m_k}W_{P/Q}^n\cdots.
\end{align*}
Therefore we conclude $\left[0.B_{n,m_j}\cdots B_{n,m_k}W_{P/Q}^n\right]<\left[0.B_{n,m_1}\cdots B_{n,m_k}W_{P/Q}^n\right]$.
\end{proof}

\section{Proofs of Theorems A, B and C}\label{sec:Proofs}

In this section, we provide the proofs of Theorems A, B, and C. Assume that $\frac{P}{Q},\frac{a}{b},\frac{A}{B},\frac{S}{T}$ are rational numbers that satisfy \eqref{eq:MainHyp}. Consider the principal bulb $H(P/Q)$. If $\Theta$ and $\Theta'$ denote the characteristic external angles landing at the root of $H(P/Q)$, then by Proposition~\ref{prop:ExtAngles} they can be expressed as
\[\Theta=\tzo(P/Q)=\overline{0.\beta_1\cdots\beta_{Q-2}01}\qquad\text{and}\qquad\Theta'=\toz(P/Q)=0.\overline{\beta_1\cdots\beta_{Q-2}10}\] $\Theta<\Theta'$. Now consider the case of external angles associated with broken lines: Fix $n\geq 1$. As shown in Lemma~\ref{lem:R1}, the M-sequence of the broken line $BL(\frac PQ,\frac ab, n)$ is given by
\begin{equation}\label{eq:bl}
\theta:=\tzo\left(\frac PQ,\frac ab, n\right)=0.\overline{B_{n,m_1}\cdots B_{n,m_k}},
\end{equation}
for some $k\geq1$ and $m_i\in\{m,m+1\}$ for some $m\geq0$. We begin with the proof of Theorem~\ref{thm:B} which shows how to compute the conjugate angle of $\theta$. Then, using kneading sequences, we prove Theorem~\ref{thm:C} and conclude with the proof of Theorem~\ref{thm:A}.

\subsection{Conjugate angles}
Our goal in this section is to 
provide an algorithm to compute the conjugate angle associated with the angle of a broken line. We begin by introducing some notation.

\begin{definition}\label{def:D2}
Given a finite word $W=\alpha_1\cdots\alpha_n$ with $\alpha_j\in\{0,1\}$ and $j=1,\ldots,n$, we denote by $W'=a_1\cdots a_n$ the word obtained by adding 1 modulo 2 with carry over to the left. More precisely, $W'=W+1$, where $a_{k}=\alpha_{k}+t_k \mod 2$, with $t_n=1$, $t_{n-j}=1$ if $a_{n-j+1}=0$ and $t_{n-j}=0$ if $a_{n-j+1}=1$, for $j=1,\dots,n-1.$ 
\end{definition}

\begin{definition}\label{def:tetas}
Fix $n\geq 1$ and let $\theta:=\tzo \left(\frac PQ,\frac ab, n\right)=0.\overline{B_{n,m_1}\cdots B_{n,m_k}}$. We define
\[\theta^{\prime}=0.\overline{B_{n,m_1}'\cdots B_{n,m_k}'}.
\]
Regarding $\theta$ as the binary sequence  $\theta=0.\overline{\alpha_1\cdots\alpha_b}$ we will express $\theta^{\prime}$ as a binary sequence  $\theta'=0.\overline{\alpha_1'\cdots\alpha_b'}$. Furthermore,  for each $k=1,\ldots,b$, we define
\begin{equation}\label{eq:tetak}
\theta_k=0.\alpha_{b-k+1}'\cdots\alpha_b'\overline{\alpha_1\cdots\alpha_b}.
\end{equation}
\end{definition}

The prime notation defined above has several consequences. From now on, we adopt the convention $W_{1/1}^{01}:=1$.

\begin{lemma}\label{lem:Lem2}
Let $\frac{a}{b}$ and $\frac{c}{d}$ be Farey neighbors such that $0<a/b<c/d\leq 1$ and let $p/q$ be their mediant. Given the associated words $W_{a/b}$, $W_{c/d}$ and $W_{p/q}$, one has
\[W_{p/q}=W_{c/d}W_{a/b}=W_{a/b}'W_{c/d}.\]
\end{lemma}

\begin{proof}
The left-hand equality follows from Proposition~\ref{prop:P1}. We show $W_{c/d}W_{a/b}=W_{a/b}'W_{c/d}$. First, observe that for $n\in\mathbb N$, if $\frac{a}{b}=\frac{1}{n+1}$ and $\frac{c}{d}=\frac{1}{n}$, then $W_{c/d}=0^{n-1}1$, $W_{a/b}=0^n1$, and $W_{a/b}'=0^{n-1}10$. Therefore, $W_{c/d}W_{a/b}=0^{n-1}10^n1=W_{a/b}'W_{c/d}$, as desired. We now argue inductively. We have two cases. 
\begin{enumerate}
\item Assume that $\frac{a}{b}<\frac{s}{t}$ are Farey neighbors,  $\frac{c}{d}=\frac{a+s}{b+t}$ and suppose the lemma holds true for $\frac{c}{d}$, that is, $W_{c/d}=W_{s/t}W_{a/b}=W_{a/b}^{\prime}W_{s/t}$. Hence
\[ W_{p/q}=W_{c/d}W_{a/b}=W_{a/b}'W_{s/t}W_{a/b}=W_{a/b}^{\prime}W_{c/d},
\]
as desired.

\item Assume that $\frac{s}{t}<\frac{c}{d}$ are Farey neighbors,  $\frac{a}{b}=\frac{c+s}{d+t}$ and suppose the lemma holds  true for $\frac{a}{b}$, that is, $W_{a/b}=W_{c/d}W_{s/t}=W_{s/t}'W_{c/d}$. Note that  $W_{a/b}'=W_{c/d}W_{s/t}^{\prime}$, and hence
\[ W_{p/q}=W_{c/d}W_{a/b}=W_{c/d}W_{s/t}'W_{c/d}=W_{a/b}'W_{c/d}.
\]

\end{enumerate}
As every pair of Farey neighbors belongs to a Farey series and the Farey series are obtained by means of the mediant sum, the result follows.
\end{proof}

Let $\theta=0.\overline{\alpha_1\cdots \alpha_b}= 0.\overline{B_{n,m_1}\cdots B_{n,m_k}}$ be the M-sequence (with $01$-convention) of the broken line $BL(\frac PQ, \frac ab, n)$ as in (\ref{eq:bl}) and consider the binary periodic sequence given in Definition~\ref{def:D2}, namely
\[\theta'=0.\overline{\alpha_1'\cdots\alpha_b'}=0.\overline{B_{n,m_1}'\cdots B_{n,m_k}'}.\]
To show that $\theta'$ is the conjugate external angle of $\theta$, we prove in the next lemmas, that $\theta'$ satisfies Lemma~13.1 (Conjugate External Angle) and Lemma~13.3 (Conjugate External Angle Algorithm) in \cite{BS}. The following definitions and terminology are also taken from Section 13 in \cite{BS}. Given two pairs of distinct points $\{\theta, \theta'\},\{\gamma, \gamma'\}\subset \mathbb{S}^1=\bbR/\bbZ$, they are said to be {\em unlinked} if the pair $\{\gamma, \gamma'\}$ lies in a single connected component of $\mathbb{S}^1\setminus \{\theta,\theta'\}$ (or conversely, if $\{\theta, \theta'\}$ lies in a single connected component of $\mathbb{S}^1\setminus \{\gamma,\gamma'\}$).

We first show that intervals $(\theta_1,2^{b-1}\theta\mod 1)$ and $(2^{b-k}\theta\mod 1,\theta_k)$ are unlinked under the hypothesis $0<\frac AB < \frac ST<1$ over the Farey parents of $\frac PQ$.

\begin{lemma}\label{lem:LemP}
Fix $n\geq 1$ and let $\frac PQ, \frac ab$ be rational values that satisfy (\ref{eq:MainHyp}). Also, assume that $0<\frac AB < \frac ST<1$. Let $\theta$, $\theta'$ and $\theta_k$ be as in Definition~\ref{def:tetas}. Then, for any $k=2,\ldots,b$,
\begin{enumerate}
\item if $\alpha_{b-k+2}$ is the first digit of $W_{P/Q}^sW_{S/T}$ for some $0\leq s<n$, then
\[\theta_1<2^{b-1}\theta\mod 1<2^{b-k}\theta\mod 1<\theta_k,\]
\item otherwise, $\theta_1<\theta_k,2^{b-k}\theta\mod 1<2^{b-1}\theta\mod 1$.
\end{enumerate}
In both cases, the intervals $(\theta_1,2^{b-1}\theta\mod 1)$ and $(2^{b-k}\theta\mod 1,\theta_k)$ are unlinked.
\end{lemma}

\begin{remark}
The hypothesis $0<\frac{A}{B}<\frac{P}{Q}<\frac{S}{T}<1$ implies that $\frac PQ\neq \frac 1n, \frac{n}{n+1}$ for all $n\in \bbN$. We derive the next observations that will be important in the proof of the lemma. 
\begin{itemize}
\item[(i)] if $\frac{P}{Q}<\frac{1}{2}$, then there exists $N\geq2$ such that
\[\frac{1}{N+1}\leq \frac AB<\frac{P}{Q}<\frac{S}{T}\leq\frac{1}{N},\]
then Corollary~\ref{cor:Cor1} and Lemma~\ref{lem:Lem2} imply that $W_{P/Q}$ and $W_{S/T}$ are concatenations of words of the form $W_{1/n}=0^{n-1}1$ for $n\in \{N, N+1\}$.
\item[(ii)] If $\frac{P}{Q}>\frac{1}{2}$, then there exists $M\in\mathbb{N}$ such that
\[\frac{M}{M+1}\leq \frac AB<\frac{P}{Q}<\frac{S}{T}\leq\frac{M+1}{M+2},\]
and again Corollary~\ref{cor:Cor1} shows that $W_{P/Q}$ and $W_{S/T}$ are concatenations of $W_{m/m+1}=1^{m-1}01$ for $m\in\{M, M+1\}$.
\item[(iii)] For any $k=2,\ldots,b$ one has
\[2^{b-k}\theta\mod 1=0.\alpha_{b-k+1}\cdots\alpha_b\overline{\alpha_1\cdots\alpha_b}\]
and since $\theta_k=0.\alpha_{b-k+1}'\cdots\alpha_b'\overline{\alpha_1\cdots\alpha_b}$, then the arc $(\theta_1,2^{b-1}\theta\mod 1)$ in $\mathbb{S}^1$ is given by
\[(\theta_1,2^{b-1}\theta\mod 1)=(0.0\overline{\alpha_1\cdots\alpha_b},0.1\overline{\alpha_1\cdots\alpha_b})=\left(\frac{\theta}{2},\frac{\theta+1}{2}\right).\]

\item[(iv)] Since $\theta=0.\overline{\alpha_1\cdots\alpha_b}=0.\overline{W_{P/Q}^n\alpha_{nQ+1}\cdots\alpha_b}$, then $\tzo(a/b)=0.\overline{\alpha_{nQ+1},\cdots\alpha_b W_{P/Q}^n}$ is the M-sequence associated with $\frac{a}{b}$. Writing $\tzo(a/b)=0.\overline{W_{\zeta_1}\cdots W_{\zeta_k}}$, with $\zeta_j\in \{P/Q, S/T\}$, then $\zeta_j=P/Q$ for $j=k-n+1,\ldots,k$ and thus $\theta=0.\overline{W_{\zeta_{k-n+1}}\cdots W_{\zeta_k}W_{\zeta_1}\cdots W_{\zeta_{k-n}}}$. Part (3) of Lemma~\ref{lem:Lemy0} implies that $\tzo(a/b)>\theta$. A dynamical argument of external rays also show that $\frac{1}{2}\tzo(a/b)<2^m\tzo(a/b) \mod 1$ for all $m$. Since $2^m \tzo(a/b)=2^{m+nQ}\theta \mod 1$ then for all $0\leq k\leq b-1$,
\[\theta_1=\frac{\theta}{2}<\frac{\tzo(a/b)}{2}<2^k\theta\mod 1.\]

\item[(v)] If $\alpha_{b-k+1}'=0$ and $\alpha_{b-k+1}=1$ then $\theta_k<2^{b-k}\theta\mod 1$. Otherwise, $\theta_k>2^{b-k}\theta\mod 1$.
\end{itemize}
\end{remark}

\begin{proof}[Proof of Lemma~\ref{lem:LemP}]
Let $\alpha_{b-k+2}$ be the first digit of $W_{P/Q}^sW_{S/T}$, for some $0\leq s<n$. Therefore, $\alpha_{b-k+1}$ is the last digit of a word $W_{S/T}$ or $W_{P/Q}$, and due to the $01$-convention, $\alpha_{b-k+1}'=\alpha_{b-k+1}=1$. Without loss of generality, we can assume that $\frac{P_m}{Q_m}<\frac{a}{b}\leq\frac{P_{m+1}}{Q_{m+1}}$ for a given $m\geq0$. We divide the proof into several subcases depending on the values of $n$ (the hinge point), $s$ and $m$.
\begin{enumerate}
    \item If $n=1$ then $s=0$ and hence $\alpha_{b-k+2}$ is the first element of some block $W_{S/T}$. Therefore, $0.\alpha_{b-k+1}\alpha_{b-k+2}\cdots=0.1W_{S/T}^rW_{P/Q}\cdots$ with $1\leq r\leq m+1$. The Lemma~\ref{lem:Lemy0} allows us to compare the following expressions
    \begin{align*} 
    2^{b-1}\theta\mod 1&=0.1B_{n,m+1}\cdots=0.1W_{P/Q}W_{S/T}^{m+1}\cdots\in \left[0.1W_{P/Q}W_{S/T}^r\right]\\
                 &<\left[0.1W_{S/T}^rW_{P/Q}\right]\ni0.\alpha_{b-k+1}\alpha_{b-k+2}\cdots=2^{b-k}\theta\mod 1. 
    \end{align*}
    \item Now, assume $n>1$ and $s=0$. In this case $0.\alpha_{b-k+1}\alpha_{b-k+2}\cdots=0.1W_{S/T}W_{P/Q}^{n-1}\cdots$. Therefore, by Lemma \ref{lem:Lemy0} we have that     \begin{align*} 
    2^{b-1}\theta\mod 1&=0.1B_{n,m+1}\cdots\in\left[0.1W_{P/Q}^nW_{S/T}\right]<\left[0.1W_{P/Q}^{n-1}W_{S/T}W_{P/Q}\right]\\
                 &<\left[0.1W_{S/T}W_{P/Q}^{n-1}\right]\ni 0.\alpha_{b-k+1}\alpha_{b-k+2}\cdots=2^{b-k}\theta\mod 1. 
    \end{align*}
    \item Finally, assume $n>1$ and $s\geq1$. In this case 
\[0.\alpha_{b-k+1}\alpha_{b-k+2}\cdots=0.1W_{P/Q}^sW_{S/T}W_{P/Q}^{n-1}\cdots=0.1W_{P/Q}^sW_{S/T}W_{P/Q}^{n-s}\cdots.\]
Lemma \ref{lem:Lemy0} implies \begin{align*} 
    2^{b-1}\theta\mod 1&=0.1B_{n,m+1}\cdots=0.1W_{P/Q}^nW_{S/T}\cdots\in \left[0.1W_{P/Q}^sW_{P/Q}^{n-s}W_{S/T}\right]\\
                 &<\left[0.1W_{P/Q}^sW_{S/T}W_{P/Q}^{n-s}\right]\ni0.\alpha_{b-k+1}\alpha_{b-k+2}\cdots=2^{b-k}\theta\mod 1. 
    \end{align*}
\end{enumerate}
We conclude that $\theta_1<2^{b-1}\theta\mod 1<2^{b-k}\theta\mod 1<\theta_k$.

\item Assume that $\alpha_{b-k+2}$ is not the first digit of some word $W_{P/Q}^sW_{S/T}$ for all $0\leq s<n$. We prove that $\theta_1<\theta_k,2^{b-k}\theta\mod 1<2^{b-1}\theta\mod 1$ by studying the four possibilities arising from $\alpha_{b-k+1},\alpha_{b-k+1}'\in \{0,1\}$.

\emph{Case 1.} If $\alpha_{b-k+1}=\alpha_{b-k+1}'=0$ then
\[\theta_1<2^{b-k}\theta\mod 1<\theta_k=0.0\alpha_{b-k+2}\cdots\alpha_b\overline{\alpha_1\cdots\alpha_b}<0.1\overline{\alpha_1\cdots\alpha_b}=2^{b-1}\theta\mod 1.\]

\emph{Case 2.} If $\alpha_{b-k+1}=1$ and $\alpha_{b-k+1}'=0$, then $\alpha_{b-k+1}$ is the last term of some block $B_{n,m}$ as in (\ref{def:D1}). This implies that 
\[0.\alpha_{b-k+1}\cdots\alpha_b\overline{\alpha_1\cdots\alpha_b}=0.1\overline{B_{n,m_{j+1}}\cdots B_{n,m_j}},\]
for some $j>1$. By item 2 in Corollary \ref{cor:Lem1} we have that
\begin{eqnarray*}
2^{b-k}\theta\mod 1&=&0.\alpha_{b-k+1}\cdots\alpha_b\overline{\alpha_1\cdots\alpha_b}\in\left[0.1{B_{n,m_{j+1}}\cdots B_{n,m_j}}\right]\\
&<&\left[0.1{B_{n,m_1}\cdots B_{n,m_T}} \right]\ni 2^{b-1}\theta\mod 1.
\end{eqnarray*}

To see $\theta_k>\theta_1$, note that since $k>1$, then $\frac{a}{b}\neq\frac{P_m}{Q_m}$ for all $m\in\mathbb{N}$, therefore $\frac{P_m}{Q_m}<\frac{a}{b}<\frac{P_{m+1}}{Q_{m+1}}$ for some $m\geq 0$. Since $\alpha_{b-k+1}=1$ and $\alpha_{b-k+1}'=0$, then
\[\theta_k=0.\alpha_{b-k+1}'\cdots\alpha_b'\overline{\alpha_1\cdots\alpha_b}=0.0B_{n,m_i}'\cdots,\]
where $m_i\in\{m,m+1\}$. Note that $\left[0.B_{n,m}'\right], \left[0.B_{n,m+1}'\right]>\left[0.B_{n,m+1}\right]$, therefore 
\[\theta_k=0.\alpha_{b-k+1}'\cdots\alpha_b'\overline{\alpha_1\cdots\alpha_b}\in\left[0.0B_{n,m_i}'\right]>\left[0.0B_{n,m+1}\right]\ni0.0\overline{\alpha_1\cdots\alpha_b}=\theta_1.\]
Then $\theta_1<\theta_k<2^{b-k}\theta\mod 1<2^{b-1}\theta\mod 1.$

\emph{Case 3.} If $\alpha_{b-k+1}=0$ and $\alpha_{b-k+1}'=1$, then, $\theta_k=0.\alpha_{b-k+1}'\cdots\alpha_b'\overline{\alpha_1\cdots\alpha_b}=0.10B_{n,m_i}'\cdots$ with $m_i\in\{m,m+1\}$ or $\theta_k=0.10B_{n,m+1}\cdots$ if $k=2$. Note that, if $\frac{P}{Q}<\frac{1}{2}$ (respectively if $\frac{P}{Q}>\frac{1}{2}$), then 
\[\theta_k=0.\alpha_{b-k+1}'\cdots\alpha_b'\overline{\alpha_1\cdots\alpha_b}\in\left[0.10W_{\zeta_1}\right]<\left[0.1W_{\zeta_1}\right]\ni0.1\overline{\alpha_1\cdots\alpha_b}=2^{b-1}\theta\mod 1,\]
with $\zeta_1\in \{\frac 1N, \frac{1}{N+1}\}$ (respectively $\zeta_1\in \{\frac{M}{M+1}, \frac{M+1}{M+2}\}$). Therefore, $\theta_1<2^{b-k}\theta<\theta_k<2^{b-1}\theta$.

\emph{Case 4.} If $\alpha_{b-k+1}=\alpha_{b-k+1}'=1$, then $\alpha_{b-k+1}$ is neither the last nor the second-to-last digit of any block $B_{n,m}$ or $B_{n,m+1}$. Consequently, $\alpha_{b-k+2}$ is not the first or the last digit of any block $B_{n,m}$ or $B_{n,m+1}$. Recalling that for $m\geq 1$,
\[B_{n,m}=W_{P/Q}^n\left(W_{S/T}W_{P/Q}^{n-1}\right)^{m-1}W_{S/T},\qquad B'_{n,m}=W_{P/Q}^n\left(W_{S/T}W_{P/Q}^{n-1}\right)^{m-1}W_{S/T}'\]
and that $B'_{n,0}=W_{P/Q}'$, we derive the following cases: $\alpha'_{b-k+2}$ is a middle digit of
\begin{enumerate}
\item[(a)] a $W_{P/Q}$ word,
\item[(b)] a $W_{S/T}$ word,
\item[(c)] a $W_{P/Q}'$ word (this is possible only when $m=0$), or
\item[(d)] a $W_{S/T}'$ word.
\end{enumerate}
We prove that $\theta_1<\theta_k,2^{b-k}\theta<2^{b-1}\theta$ in each of these cases. 
\begin{enumerate}
\item[(a)] Assume first that $\frac{1}{N+1}<\frac PQ<\frac 1N$ for some $N\geq 2$. Then $W_{P/Q}=W_{\zeta_1}\cdots W_{\zeta_L}$ for some $L>1$ and with $\zeta_j\in \{\frac 1N, \frac{1}{N+1}\}$. Since $\alpha_{b-k+1}=1$, then it must be the last digit of a word $W_{\zeta_j}=0^l1$, for some $l\in \{N-1,N\}$ and $1\leq j<L$, that is
\[\theta_k=0.1\alpha_{b-k+2}'\alpha_{b-k+3}'\cdots=0.1W_{\zeta_{j+1}}\cdots W_{\zeta_L}\cdots.\]
Since {$\left[0.1W_{\zeta_{j+1}}\cdots W_{\zeta_L}\right]<\left[0.1W_{\zeta_1}\cdots W_{\zeta_L}\right]$} by Lemma~\ref{lem:Lemy0}, then $\theta_k<2^{b-1}\theta\mod 1=0.1W_{\zeta_1}\cdots W_{\zeta_L}\cdots$.

Now, if $\frac{M}{M+1}<\frac{P}{Q}<\frac{M+1}{M+2}$ for some $M\geq 1$, then $W_{P/Q}=W_{\zeta_1}\cdots W_{\zeta_L}=1^M01\cdots$ for $\zeta_j\in \left\{\frac{M}{M+1}, \frac{M+1}{M+2}\right\}$. The assumption $\alpha'_{b-k+1}=1$ implies that
\[\theta_k=0.\alpha'_{b-k+1}\cdots\alpha'_b\overline{\alpha_1\cdots\alpha_b}=0.1^s0\cdots,\]
for some $1\leq s\leq M+1$. If $s\leq M$, then $0.1^s0\cdots<0.1^{M+1}0\cdots$ hence 
$\theta_k\in\left[0.1^s0 \right]< \left[0.1 W_{P/Q}\right]\ni2^{b-1}\theta\mod 1$. The case when $s=M+1$ only occurs when 
$\alpha_{b-k+1}=1$ is the last digit of a word $W_{\zeta_j}=1^l01$ for some $l\in\{M-1, M\}$ and some $j=1,\ldots, L-1$. Then $\theta_k=0.1W_{\zeta_{j+1}}\cdots W_{\zeta_L}\cdots$. Lemma~\ref{lem:Lemy0} shows that $\left[0.1W_{\zeta_{j+1}}\cdots W_{\zeta_L}\right]<\left[0.1W_{P/Q}\right]$ and hence $\theta_k<2^{b-1}\theta\mod 1$.  Therefore, $\theta_1<2^{b-k}\theta\mod 1<\theta_k<2^{b-1}\theta\mod 1$.

\item[(b)] The assumption on $\alpha_{b-k+2}'$ implies that $\frac{1}{N+1}<\frac{S}{T}<\frac{1}{N}$ for some $N\in\mathbb{N}$. Since $W_{P/Q}=W_{S/T}W_{A/B}$, then
$2^{b-1}\theta\mod 1=0.1W_{P/Q}\cdots=0.1W_{S/T}\cdots$, so the proof of this case is analogous to case (a) after replacing $W_{P/Q}$ with $W_{S/T}$.

\item[(c)] In this case $m=0$, thus $\theta$ is a concatenation of blocks of the form $B_{n,0}=W_{P/Q}$ and $B_{n,1}=W_{P/Q}^nW_{S/T}$. Then the preperiodic part of $\theta_k$ is the concatenation of $B'_{n,0}=W'_{P/Q}$ and $B'_{n,1}=W_{P/Q}^nW'_{S/T}$. Proposition~\ref{prop:concat1} shows that $W_{P/Q}=W_{S/T}W_{A/B}$ and hence $W_{P/Q}'=W_{S/T}W_{A/B}'$ since $0<\frac AB$. The location of $\alpha'_{b-k+2}$ generates three subcases.
\begin{enumerate}
\item[(i)] Assume that $\alpha_{b-k+2}'$ is a middle digit of $W_{S/T}$ (as the first word of $W_{P/Q}'$). This case is discussed in (b).

\item[(ii)] Assume that $\alpha_{b-k+2}'$ is a middle digit of $W_{A/B}'$ (as the second word in $W'_{P/Q}$). The word $W'_{P/Q}$ is followed by either $W'_{P/Q}$ or $W_{P/Q}$, in both cases, $W'_{P/Q}$ is followed by $W_{S/T}$. Lemma~\ref{lem:Lem2} shows that
\[W_{P/Q}'W_{S/T}=W_{S/T}W_{A/B}'W_{S/T}=W_{S/T}W_{P/Q},\]
therefore, we can assume that $\alpha_{b-k+2}'$ is a middle element of a block $W_{P/Q}$. This case is discussed in (a).

\item[(iii)] If $\alpha_{b-k+2}'$ is the first digit of $W_{A/B}'$ then for some integer $r\geq0$,
\begin{align*}
   \theta_k=0.1 \alpha_{b-k+2}'\cdots\alpha_b'\overline{\alpha_1\cdots\alpha_b}&=0.1W_{A/B}'(W_{P/Q}')^rW_{P/Q}\cdots\\
    &=0.1W_{A/B}'(W_{S/T}W_{A/B}')^rW_{S/T}W_{A/B}\cdots\\
    &=0.1W_{P/Q}^{r+1}W_{A/B}\cdots
\end{align*}
The Lemma \ref{lem:L1} shows that $\left[0.W_{A/B}\right]<\left[0.W_{P/Q}\right]<\left[0.W_{S/T}W_{P/Q}\right]$  and hence \[\theta_k\in\left[0.1W_{P/Q}^{r+1}W_{A/B}\right]<\left[0.1W_{P/Q}^nW_{S/T}W_{P/Q}\right],\]
regardless of the value of $r$, thus  $\theta_k<0.1W_{P/Q}^nW_{S/T}W_{P/Q}\cdots=2^{b-1}\theta\mod 1$.
\end{enumerate}
In all the cases, $\theta_1<2^{b-k}\theta\mod 1<\theta_k<2^{b-1}\theta\mod 1$.

\item[(d)] Assume that $\alpha_{b-k+2}'$ is a middle digit of a word $W_{S/T}'$. Since $\alpha'_{b-k+1}=1$, then there exists $N\in\mathbb{N}$ such that $\frac{1}{N+1}<\frac{S}{T}<\frac{1}{N}$. Let $0<M<P$ be the Farey parents of $S/T$ and denote by $W_M$ and $W_P$ the associated words of the M-sequences for $M$ and $P$  respectively. Since $\frac{S}{T}\neq\frac{1}{N}$ for all $N\in\mathbb{N}$, then by Lemma \ref{lem:Lem2}, $W_{S/T}=W_PW_M=W_M'W_P$. 
Note that $0<M\leq\frac{A}{B}<\frac ST$, since $\frac AB$ and $\frac ST$ are Farey neighbors and moreover $0.W_M\cdots \leq0.W_{A/B}\cdots$. Furthermore,
\[2^{b-1}\theta\mod 1=0.1W_{P/Q}\cdots=0.1W_{S/T}W_{A/B}\cdots=0.1W_PW_MW_{A/B}\cdots=0.1W_P\cdots.\]

The location of $\alpha'_{b-k+2}$ in $W'_{S/T}=W_PW'_M$ generates three subcases:
\begin{enumerate}
\item[(i)] If $\alpha_{b-k+2}'$ is a middle digit in $W_P$. Then $P\neq\frac{1}{n}$ for all $n\in\mathbb{N}$. This case is analogous to (a), replacing $W_{P/Q}$ with $W_P$.

\item[(ii)] Let $\alpha_{b-k+2}'$ be a middle digit in $W_M'$. The word $W_{S/T}'$ is followed by either $W_{P/Q}$ or $W_{P/Q}'$, in any case, $W_{S/T}'$ is followed by $W_{S/T}$. Note also that
\[W_{S/T}'W_{S/T}=W_PW_M'W_PW_M=W_PW_{S/T}W_M.\]
Therefore we can assume that $\alpha'_{b-k+2}$ is a middle digit in $W_{S/T}$. This case is equal to (b).

\item[(iii)] Assume that $\alpha_{b-k+2}'$ is the first digit in $W_M'$. Recall $M\leq \frac AB$. If the inequality is strict, then $0.W_M\cdots <0.W_{A/B}\cdots $ and thus
\begin{align*}
\theta_k&=0.1\alpha_{b-k+2}'\cdots\alpha_b'\overline{\alpha_1\cdots\alpha_b}=0.1W_M'W_{S/T}\cdots=0.1W_M'W_P W_M\cdots\\
&=0.1W_{S/T} W_M\cdots<0.1W_{S/T}W_{A/B}\cdots=2^{b-1}\theta\mod 1.
\end{align*}
The case when $M=\frac AB$ follows from (iii) in case (c).
\end{enumerate}
In all the cases, $\theta_1<2^{b-k}\theta<\theta_k<2^{b-1}\theta$.
\end{enumerate}
\end{proof}

The next lemma establishes the unlinked result under the hypothesis $0= \frac AB < \frac ST<1$ over the Farey parents of $\frac PQ$.

\begin{lemma}\label{lem:LemPEsp1}
Assume that $\frac{P}{Q},\frac{a}{b},\frac{A}{B},\frac{S}{T}$ are rational numbers that satisfy \eqref{eq:MainHyp}, assume also that $\frac{A}{B}=\frac{0}{1}=0$ and that $\frac{S}{T}<1$. Let $\theta$, $\theta'$ and $\theta_k$ be as in Definition~\ref{def:tetas}, then for any $k=2,\ldots,b$,
\begin{enumerate}
\item if $\alpha_{b-k+2}$ is the first digit of $W_{P/Q}^sW_{S/T}$ for some $0\leq s<n$, then $$\theta_1<2^{b-1}\theta<2^{b-k}\theta<\theta_k,$$
\item otherwise, we have $\theta_1<\theta_k,2^{b-k}\theta<2^{b-1}\theta$.
\end{enumerate}
In both cases, the intervals $(\theta_1,2^{b-1}\theta)$ and $(2^{b-k}\theta,\theta_k)$ are unlinked.
\end{lemma}

\begin{remark}\label{rem:RLemPEsp1}
For the case when $\frac{A}{B}=0$ and $\frac{S}{T}<1$, it follows that $\frac{S}{T}=\frac{1}{N}$ and $\frac{P}{Q}=\frac{1}{N+1}$ for some $N\geq2$. Hence, $W_{P/Q}=0^N1$ and $W_{S/T}=0^{N-1}1$. As in Lemma \ref{lem:LemP}, we have that 
\begin{itemize}
\item $2^{b-k}\theta=0.\alpha_{b-k+1}\cdots\alpha_b\overline{\alpha_1\cdots\alpha_b}$, $\theta_k=0.\alpha_{b-k+1}'\cdots\alpha_b'\overline{\alpha_1\cdots\alpha_b}$, and $$(\theta_1,2^{b-1}\theta)=(0.0\overline{\alpha_1\cdots\alpha_b},0.1\overline{\alpha_1\cdots\alpha_b})=\left(\frac{\theta}{2},\frac{\theta+1}{2}\right).$$
\item For all $2\leq k\leq b-1$ we have that $\theta_1<2^{b-k}\theta$.
\item If $\alpha_{b-k+1}'=0$ and $\alpha_{b-k+1}=1$ then $\theta_k<2^{b-k}\theta$. Otherwise, $\theta_k>2^{b-k}\theta$.
\end{itemize}
\end{remark}

\begin{proof}
Now we prove Lemma~\ref{lem:LemPEsp1}.
\begin{enumerate}
\item The proof of this statement is equal to the proof of 1. in Lemma \ref{lem:LemP}.
\item Assume that $\alpha_{b-k+2}$ is not the first digit of $W_{P/Q}^sW_{S/T}$, for all $0\leq s<n$. We prove that $\theta_1<\theta_k,2^{b-k}\theta\mod 1<2^{b-1}\theta\mod 1$ by studying the four possibilities arising from $\alpha_{b-k+1},\alpha_{b-k+1}'\in \{0,1\}$.

\emph{Case 1.} If $\alpha_{b-k+1}=\alpha_{b-k+1}'=0$ then the proof of Case 1 in Lemma \ref{lem:LemP} remains valid.

\emph{Case 2.} If $\alpha_{b-k+1}=1$ and $\alpha_{b-k+1}'=0$ then the proof of Case 2 in Lemma \ref{lem:LemP} remains valid.

\emph{Case 3.} If $\alpha_{b-k+1}=0$ and $\alpha_{b-k+1}'=1$ then, $\theta_k=0.\alpha_{b-k+1}'\cdots\alpha_b'\overline{\alpha_1\cdots\alpha_b}=0.10B_{n,m_i}'\cdots$ with $m_i\in\{m,m+1\}$, or $\theta_k=0.10B_{n,m+1}\cdots$ if $k=2$. Note that, if $m_i>0$ or if $k=2$, then these cases reduce to the proof of Case 3 in Lemma \ref{lem:LemP}. If $k>2$ and $m_i=0$ then
\[\theta_k=0.\alpha_{b-k+1}'\cdots\alpha_b'\overline{\alpha_1\cdots\alpha_b}=0.10(W_{P/Q}')^r W_{P/Q}\cdots=0.1(W_{P/Q})^r0 W_{P/Q}\cdots,\]
for some $r\geq1$. Since
$[0.0W_{P/Q}] <[0.W_{P/Q}]<[0.W_{S/T}W_{P/Q}]$, then $[0.(W_{P/Q})^k 0 W_{P/Q}]<[0.(W_{P/Q})^{k'} W_{S/T}W_{P/Q}]$ for all $k, k'\geq 0$ and hence
\[\theta_k\in [0.1(W_{P/Q})^r0W_{P/Q}]<[0.1W_{P/Q}^n W_{S/T}W_{P/Q}]\ni 2^{b-1}\theta\mod 1.\]
Therefore, $\theta_1<2^{b-k}\theta\mod 1<\theta_k<2^{b-1}\theta\mod 1$

\emph{Case 4.} The case in which $\alpha_{b-k+1}=1$ and $\alpha_{b-k+1}'=1$ only happens when $\alpha_{b-k+1}'$ is the last digit of some word $W_{P/Q}$ or $W_{S/T}$, but it is not the last digit of some block $B_{n,m}$. This implies that $\alpha_{b-k+2}$ is the first digit of $W_{P/Q}^sW_{S/T}$, for some $0\leq s<n$. This case is already contemplated in the first part of this lemma.
\end{enumerate}

\end{proof}

We conclude with the last lemma that establishes the unlinked result under the hypothesis $0\leq \frac AB < \frac ST = 1$ over the Farey parents of $\frac PQ$.

\begin{lemma}\label{lem:LemPEsp2}
Assume that $\frac{P}{Q},\frac{a}{b},\frac{A}{B},\frac{S}{T}$ are rational numbers that satisfy \eqref{eq:MainHyp} and that $\frac{S}{T}=1$, let $\theta$, $\theta'$ and $\theta_k$ be as in Definition~\ref{def:tetas}.  For any $k=2,\ldots,b$,
\begin{enumerate}
\item if $\alpha_{b-k+2}$ is the first digit of $W_{P/Q}^sW_{S/T}$ for some $0\leq s<n$, then 
\begin{enumerate}
    \item if $n=1$ then $\theta_k<\theta_1<2^{b-1}\theta<2^{b-k}\theta$.
    \item if $n>1$ and $\alpha_{b-k+2}$ corresponds to the last word $W_{S/T}$ of some block $B_{n,m}$, for some $m\geq1$, then $\theta_k<\theta_1<2^{b-1}\theta<2^{b-k}\theta$.
    \item if $n>1$ and $\alpha_{b-k+2}$ does not correspond to the last word $W_{S/T}$ of some block $B_{n,m}$, for some $m\geq1$, then $\theta_1<2^{b-1}\theta<2^{b-k}\theta<\theta_k.$
\end{enumerate}
\item otherwise, we have $\theta_1<\theta_k,2^{b-k}\theta<2^{b-1}\theta$.
\end{enumerate}
In both cases, the intervals $(\theta_1,2^{b-1}\theta)$ and $(2^{b-k}\theta,\theta_k)$ are unlinked.
\end{lemma}

\begin{remark}\label{rem:RLemPEsp2}
For the case when $\frac{S}{T}=1$ and $0\leq \frac AB$ we can find $N\geq0$ so that $\frac{A}{B}=\frac{N}{N+1}$ and $\frac{P}{Q}=\frac{N+1}{N+2}$. Therefore $W_{P/Q}=1^N01$ and $W_{S/T}=1$. Moreover, 
\begin{itemize}
    \item if $n=1$, then $B_{n,m}=W_{P/Q}W_{S/T}^m$, and $B_{n,m}'=W_{P/Q}'0^m=1^{N+1}0^{m+1}$.
    \item if $n>1$ and $m=1$ then $B_{n,m}=W_{P/Q}^nW_{S/T}$, and $B_{n,m}'=W_{P/Q}^{n-1}W_{P/Q}'0$.
    \item if $n>1$ and $m\geq2$ then $B_{n,m}=W_{P/Q}^n[W_{S/T}W_{P/Q}^{n-1}]^{m-1}W_{S/T}$, and
    \[B_{n,m}'=W_{P/Q}^n[W_{S/T}W_{P/Q}^{n-1}]^{m-2}W_{S/T}W_{P/Q}^{n-2}W_{P/Q}'0.\]
\end{itemize}
As in Lemma \ref{lem:LemP}, we have that 
\begin{itemize}
\item $2^{b-k}\theta=0.\alpha_{b-k+1}\cdots\alpha_b\overline{\alpha_1\cdots\alpha_b}$, $\theta_k=0.\alpha_{b-k+1}'\cdots\alpha_b'\overline{\alpha_1\cdots\alpha_b}$, and $$(\theta_1,2^{b-1}\theta)=(0.0\overline{\alpha_1\cdots\alpha_b},0.1\overline{\alpha_1\cdots\alpha_b})=\left(\frac{\theta}{2},\frac{\theta+1}{2}\right).$$
\item For all $2\leq k\leq b-1$ we have that $\theta_1<2^k\theta$.
\item If $\alpha_{b-k+1}'=0$ and $\alpha_{b-k+1}=1$ then $\theta_k<2^{b-k}\theta$. Otherwise, $\theta_k>2^{b-k}\theta$.
\end{itemize}
\end{remark}

\begin{proof}
Now we prove Lemma~\ref{lem:LemPEsp2}.
\begin{enumerate}
\item First, we prove that $2^{b-1}\theta\mod 1<2^{b-k}\theta\mod 1$. 

On the one hand, since $\alpha_{b-k+2}$ is the first element of some block $W_{P/Q}^sW_{S/T}$ with $0\leq s<n$, then
\[2^{b-k}\theta=0.\alpha_{b-k+1}\alpha_{b-k+2}\cdots\alpha_b\overline{\alpha_1\cdots\alpha_b}=0.1W_{P/Q}^sW_{S/T}^rW_{P/Q}\cdots\]
for some $r\geq1$. On the other hand, $2^{b-1}\theta\mod 1=0.1W_{P/Q}^nW_{S/T}\cdots$. Note that, since $n-s\geq1$, then
\[0.W_{P/Q}^{n-s}W_{S/T}\cdots\in[0.1^N01\cdots]<[0.1^{N+r}01\cdots]\ni 0.W_{S/T}^rW_{P/Q}\cdots,\]
for all $r\geq1$. Therefore, 
\begin{align*}
2^{b-1}\theta\mod 1&=0.1W_{P/Q}^nW_{S/T}\cdots\in[0.1W_{P/Q}^sW_{P/Q}^{n-s}W_{S/T}\cdots]\\
             &<[0.1W_{P/Q}^sW_{S/T}^rW_{P/Q}\cdots]\ni 2^{b-k}\theta\mod 1.
\end{align*}

Now, we prove the items.

\begin{enumerate}
    \item If $n=1$, then $s=0$ and hence $\alpha_{b-k+2}=W_{S/T}=1$. Note that $\alpha_{b-k+1}$ also corresponds to a block $W_{S/T}$ or it is the last digit of some word $W_{P/Q}$, in any case $\alpha_{b-k+1}=1$. Moreover, $\alpha_{b-k+1}'=\alpha_{b-k+2}'=0$, therefore
\[\theta_k=0.\alpha_{b-k+1}'\alpha_{b-k+2}'\cdots\alpha_b'\overline{\alpha_1\cdots\alpha_b}=0.0^rB_{n,m_j}'\cdots=0.0^r1^{N+1}0^{m_j+1}\cdots\]
or
\[\theta_k=0.\alpha_{b-k+1}'\alpha_{b-k+2}'\cdots\alpha_b'\overline{\alpha_1\cdots\alpha_b}=0.0^rB_{n,m+1}\cdots=0.0^r1^{N}01^{m+2}\cdots\]
with $2\leq r\leq m+2$, $m\geq0$ and $m_j\in\{m,m+1\}$. 

Recall that $\theta_1=0.0B_{n,m}\cdots=0.01^N01^{m+2}\cdots$. It is easy to check whether $r>2$ or $N>0$ then
\[\theta_k\in[0.0^r1^N\cdots]<[0.01^N01^{m+2}\cdots]\ni\theta_1.\]

Now, assume that $r=2$ and $N=0$. Then $\theta_k=0.0010\cdots$ or $\theta_k=0.0001\cdots$, in any case
\[\theta_k<0.0011\cdots=\theta_1.\]
Therefore, $\theta_k<\theta_1<2^{b-1}\theta<2^{b-k}\theta$.

\item Let $n>1$ and let $\alpha_{b-k+2}$ be the last word $W_{S/T}=1$ of some block $B_{n,m}$. In this case $\alpha_{b-k+1}$ and $\alpha_{b-k+2}$ are the two last digits of the block $B_{n,m}$, which implies that $\theta_k=0.\alpha_{b-k+1}'\alpha_{b-k+2}'\cdots=0.00B_{n,m_i}'\cdots$ with $m_i\geq0$ if $k>2$ or $\theta_k=0.00B_{n,m_j}\cdots$ with $m_j\geq1$ if $k=2$. Note that if $m_i>0$ or $k=2$ then
\[\theta_k\in[0.00W_{P/Q}\cdots]<[0.0W_{P/Q}\cdots]\ni\theta_1.\]
Otherwise, if $k>2$ and $m_i=0$, then there exists $l\geq1$ such that \[\theta_k=0.00(W_{P/Q}')^lW_{P/Q}\cdots=0.00(1^{N+1}0)^l1^N01\cdots\]
It is easy to check that if $N>0$ then $\theta_k\in[0.001^{N+1}0\cdots]<[0.01^N01\cdots]\ni\theta_1$. However, if $N=0$, then \[\theta_k=0.00(10)^l01\cdots\in[0.0(01)^l001\cdots]<[0.0(01)^n1\cdots]\ni\theta_1.\]

Therefore, $\theta_k<\theta_1<2^{b-1}\theta<2^{b-k}\theta$.

\item Assume that $n>1$ and that $\alpha_{b-k+2}$ is not the last word $W_{S/T}$ of some block $B_{n,m}$. Note that, in this case $\alpha_{b-k+1}'=\alpha_{b-k+1}$, hence $2^{b-k}\theta\mod 1<\theta_k$. Therefore, $\theta_1<2^{b-1}\theta<2^{b-k}\theta<\theta_k$.
\end{enumerate}

\item Now, we assume that $\alpha_{b-k+2}$ is not the first digit of $W_{P/Q}^sW_{S/T}$ for all $0\leq s<n$. Then we have the following three cases
\begin{enumerate}
    \item $\alpha_{b-k+2}'$ is a digit (except the first one) of some word $W_{P/Q}$.
    \item $\alpha_{b-k+2}'$ is a digit (except the first one) of some word $W_{P/Q}'$.
    \item $\alpha_{b-k+2}$ is the first digit of $W_{P/Q}^sW_{S/T}$, for some $s\geq n$.
\end{enumerate}

Now we prove each one of the previous cases.

\begin{enumerate}
    \item Assume that $\alpha_{b-k+2}'$ is a digit (except the first one) of some word $W_{P/Q}$. Then
\[\theta_k=0.\alpha_{b-k+1}'\alpha_{b-k+2}'\cdots\alpha_b'\overline{\alpha_1\cdots\alpha_b}=0.1^r01\cdots,\]
for some $0\leq r\leq N$. Hence, $\theta_k\in[0.1^r01\cdots]<[0.1^{N+1}01\cdots]\ni2^{b-1}\theta$. In this case, $\alpha_{b-k+1}=\alpha_{b-k+1}'$, and hence $\theta_1<2^{b-k}\theta\mod 1<\theta_k<2^{b-1}\theta\mod 1$.
    
\item Assume that $\alpha_{b-k+2}'$ is a digit (except the first one) of some word $W_{P/Q}'=1^{N+1}0$. Note that, $\alpha_{b-k+1}'$ is a digit of $W_{P/Q}'$ different from the last one, therefore $\alpha_{b-k+1}'=\alpha_{b-k+1}$ or $\alpha_{b-k+1}'=1$ and $\alpha_{b-k+1}=0$. In any case, $\theta_1<2^{b-k}\theta\mod 1<\theta_k$. Therefore, it only remains to be proven that $\theta_k<2^{b-1}\theta\mod 1$.
    
 If $m>0$ then there exists $1\leq r\leq N+1$ so that
\[\theta_k=0.\alpha_{b-k+1}'\alpha_{b-k+2}'\cdots\alpha_b'\overline{\alpha_1\cdots\alpha_b}\in[0.1^r00\cdots]<[0.1^{N+1}01\cdots]\ni2^{b-1}\theta\mod 1.\]

If $m=0$ then
\[\theta_k=0.\alpha_{b-k+1}'\alpha_{b-k+2}'\cdots=0.1^r0(W_{P/Q}')^lW_{P/Q}\cdots\]
for some $l\geq0$, or
\[\theta_k=0.\alpha_{b-k+1}'\alpha_{b-k+2}'\cdots=0.1^r0(W_{P/Q}')^l0\cdots\]
in both cases with $1\leq r\leq N+1$. Note that, if $r<N+1$ then
\[\theta_k\in[0.1^r0\cdots]<[0.1^{N+1}01\cdots]\ni2^{b-1}\theta\mod 1.\]

Now, let us assume that $r=N+1$. First, let us assume that
\[\theta_k=0.1^{N+1}0(W_{P/Q}')^lW_{P/Q}\cdots=0.1W_{P/Q}^l1^N01^N01\cdots\]
If $N\geq1$, then $[0.1^{N-1}01\cdots]<[0.W_{P/Q}\cdots]<[0.W_{S/T}W_{P/Q}\cdots]$, and hence
\begin{align*}
\theta_k=0.1W_{P/Q}^l1^N01^N01\cdots&\in[0.1W_{P/Q}^{l+1}1^{N-1}01\cdots]\\
&<[0.1W_{P/Q}^nW_{S/T}W_{P/Q}\cdots]\ni2^{b-1}\theta\mod 1.
\end{align*}

If $N=0$, then $\theta_k\in[0.1(01)^l001\cdots]<[0.1(01)^n101\cdots]\ni2^{b-1}\theta$. Now, let us assume that $\theta_k=0.1^{N+1}0(W_{P/Q}')^l0\cdots=0.1W_{P/Q}^l1^N00\cdots$, then 
   \begin{align*}
   \theta_k&=0.1W_{P/Q}^l1^N00\cdots\in[0.1(1^N01)^l1^N00\cdots]\\
          &<[0.1(1^N01)^n1(1^N01)\cdots]\ni0.1W_{P/Q}^nW_{S/T}W_{P/Q}\cdots\\
          &=2^{b-1}\theta\mod 1.
   \end{align*}

    Therefore, in any case we have that $\theta_1<2^{b-k}\theta<\theta_k<2^{b-1}\theta$.

    \item Assume that $\alpha_{b-k+2}'$ is the first digit of $W_{P/Q}^sW_{S/T}$, with $s\geq n$. Note that in this case, by definition of the blocks $B_{n,m}$, if $s=n$ then $\alpha_{b-k+2}$ is the first digit of some block $B_{n,m}$ and if $s>n$ then necessarily $m=0$ and $\alpha_{b-k+2}$ is the first digit of some block $ (B_{n,0})^{s-n}B_{n,1}$. In any case, $\alpha_{b-k+1}$ is the last digit of some block $B_{n,m}$, which implies that $\alpha_{b-k+1}=1$ and that $\alpha_{b-k+1}'=0$. Therefore,
 \begin{eqnarray*}
 2^{b-k}\theta\mod 1&=&0.1\overline{B_{n,m_{j+1}}\cdots B_{n,m_J}B_{n,m_1}\cdots B_{n,m_{j}}}\\
 &<& 0.1\overline{B_{n,m_1}\cdots B_{n,m_J}}=2^{b-1}\theta\mod 1.
 \end{eqnarray*}

However, since $k>1$ then
$$\theta_k=0.0(B_{n,0}')^{s-n}B_{n,1}'\cdots\hspace{5mm}\text{or}\hspace{5mm}\theta_k=0.0(B_{n,0}')^{s-n}B_{n,1}\cdots\hspace{5mm}\text{or}\hspace{5mm}\theta_k=0.0B_{n,\zeta}'\cdots$$
with $m\geq0$  and $\zeta\in\{m,m+1\}$. In any case $\theta_k=0.0B_{n,\zeta}'\cdots$ with $\zeta\in\{m,m+1\}$.

    Since $[0.B_{n,m}'\cdots], [0.B_{n,m+1}'\cdots]>[0.B_{n,m+1}\cdots]$, then 
\[\theta_k\in[0.0B_{n,\zeta}'\cdots]>[0.0B_{n,m+1}\cdots]\ni\theta_1.\]
    
    Therefore, $\theta_1<\theta_k<2^{b-k}\theta<2^{b-1}\theta$.

\end{enumerate}
\end{enumerate}
\end{proof}

Let $\theta=\tzo(P/Q, a/b,n)$ and $\theta'$ as in Definition~\ref{def:tetas}. We can now prove our second main result.

\medskip

\noindent
\textsc{Theorem B} (Characteristic angles of a broken line). Fix $n\geq 1$ and consider the fractions $\frac PQ, \frac ab$ satisfying (\ref{eq:MainHyp}). Then $\theta$ and $\theta'$ are conjugate external angles.

\begin{proof}
The choice of the $01$-convention implies that $\theta_1:=\theta/2$ is the preperiodic preimage of $\theta$. Moreover, one easily sees from (\ref{eq:tetak}) that $2\theta_{k}\mod 1=\theta_{k-1}$ for each $k$, that is, $\theta_k$ is a preimage of $\theta_{k-1}$ under angle-doubling map. Lemmas \ref{lem:LemP}, \ref{lem:LemPEsp1} and \ref{lem:LemPEsp2} show that $(\theta_1,2^{b-1}\theta)$ and $(2^{b-k}\theta,\theta_k)$ are unlinked. Therefore, the angles $\theta_k$ satisfy the hypotheses of the Conjugate External Angle Algorithm given in \cite{BS} and thus, the conjugate external angle of $\theta$ is given by
\begin{align*}
\theta+\frac{\theta_b-\theta}{1-2^{-b}}&=\theta+\frac{2^b\theta_b-2^b\theta}{2^b-1}\\
      &=\frac{2^b\theta-\theta+2^b\theta_b-2^b\theta}{2^b-1}=\frac{2^b\theta_b-\theta}{2^b-1}\\
      &=\frac{\alpha_1'\cdots\alpha_b'.\overline{\alpha_1\cdots\alpha_b}-0.\overline{\alpha_1\cdots\alpha_b}}{2^b-1}  =\frac{\alpha_1'\cdots\alpha_b'}{2^b-1}\\
      &=0.\overline{\alpha_1'\cdots\alpha_b'}=\theta'.
\end{align*}
\end{proof}

\subsection{Kneading sequences}\label{sec:kneading}

Kneading sequences were introduced by Milnor and Thurston in \cite{MR0970571} in order to study the orbit of a critical point under the iteration of a quadratic polynomial of a real variable. This concept has been successfully extended to polynomials of a complex variable; see for example \cite{A}, \cite{BS} and \cite{Kel} for further information and references.

\begin{definition}[Kneading sequence]
Let $\theta\in \bbR/\bbZ$ be given and define a partition of the unit circle given by the open arcs $\left] \frac \theta 2, \frac{\theta+1}{2}\right[$ and $\left]\frac{\theta+1}{2},  \frac \theta 2\right[$. The orbit of $\theta$ under the doubling map $D(\theta)=2\theta \mod 1$ defines its \emph{kneading sequence}, which is the sequence $K(\theta)=\Omega_1 \Omega_2 \ldots$ where
\begin{equation*}
  \Omega_i = \begin{cases}
  	\mathnormal{1} & \text{if } 2^{i-1}\theta\in \left] \frac \theta 2, \frac{\theta+1}{2}\right[,\\
	\mathnormal{0} & \text{if } 2^{i-1}\theta\in \left]\frac{\theta+1}{2},  \frac \theta 2\right[,\\
	* & \text{if } 2^{i-1}\theta\in \left\{\frac \theta 2, \frac{\theta+1}{2}\right\}.
  \end{cases} 
\end{equation*}
\end{definition}

\begin{remark}
Whenever $\theta$ is a rational angle of the form $\frac{p}{2^n-1}$ for some $p\in\{1,\ldots,2^n-2\}$ and $n\geq 2$, then its kneading sequence takes the form $K(\theta)=\overline{\Omega_1\ldots \Omega_{n-1}*}$, with $\Omega_i\in\{\mathnormal{0}, \mathnormal{1}\}$ for $i=1,\ldots,n-1$.
\end{remark}

\noindent
\textsc{Theorem C} (Kneading sequence of a broken line). 
Fix $n\geq 1$ and let $\frac{P}{Q}, \frac{a}{b},\frac{S}{T}$ be rational numbers satisfying (\ref{eq:MainHyp}). Denote the M-sequence of the broken line $BL(\frac{P}{Q},\frac{a}{b},n)$ by $\theta=0.\overline{\alpha_1\cdots\alpha_b}= 0.\overline{W_{\zeta_1}\ldots W_{\zeta_k}}$, for some $k\geq1$ and $\zeta_i\in\{S/T,P/Q\}$ for all $i=1,\ldots,k$. Let $K(\theta)=\overline{\Omega_1\cdots\Omega_{b-1}*}$ be the kneading sequence associated with $\theta$, with $\Omega_j\in \{\mathnormal{0,1}\}$ for all $i=1,\ldots,b-1$. Then, 
\begin{itemize}
    \item[(a)] $\Omega_i=\mathnormal{0}$ if and only if $\alpha_{i+1}$ is the first digit of the string $W_{P/Q}^sW_{S/T}$ for some $0\leq s<n$.
    \item[(b)] The numbers $\frac{P}{Q}$, $\frac{S}{T}$, $\frac{a}{b}$ and the M-sequence of $BL(\frac{P}{Q},\frac{a}{b},n)$ can be recovered from the kneading sequence $K(\theta)$.
\end{itemize}

\begin{proof}
\begin{itemize}
    \item[(a)] Fix any $j\in \{1,\ldots,b-1\}$. To prove sufficiency, let $\alpha_{j+1}$ be the first element of $W_{P/Q}^sW_{S/T}$, for some $0\leq s<n$. Item (1) in lemmas \ref{lem:LemP}, \ref{lem:LemPEsp1} and \ref{lem:LemPEsp2} shows that $\theta_1<2^{b-1}\theta\mod 1<2^{j-1}\theta\mod 1$, that is,
    \[\frac{\theta}{2}<\frac{\theta+1}{2}<2^{j-1}\theta\mod 1\]
    and therefore $\Omega_j=\mathnormal{0}$. To prove necessity, assume that $\alpha_{j+1}$ is not the first element of $W_{P/Q}^sW_{S/T}$. Then, from item (2) in lemmas \ref{lem:LemP}, \ref{lem:LemPEsp1} and  \ref{lem:LemPEsp2},  $\theta_1<2^{j-1}\theta\mod 1<2^{b-1}\theta\mod 1$, that is,
    \[\frac{\theta}{2}<2^{j-1}\theta\mod 1<\frac{\theta+1}{2},\]
    therefore $\Omega_j=\mathnormal{1}$.
    \item[(b)] Let us assume that we know the kneading sequence $K(\theta)$. Note that $K(\theta)$ is formed by blocks of the form $$\underbrace{\mathnormal{1\cdots1}}_{n\text{ digits}}\mathnormal{0}\quad\text{ and }\quad\underbrace{\mathnormal{1\cdots1}}_{m\text{ digits}}*$$with $m,n\geq0$.
    We  algorithmically recover the numbers $P,Q,S,T,a,b$ and $n$:
    \begin{enumerate}
        \item Since $\theta$ is associated with a broken line, then $K(\theta)$ is periodic for some period, this period is $b$.
        \item The length of the first block is $Q$.
        \item $K(\theta)$ starts with many blocks of length $Q$ and then a block of length different than $Q$. The number of blocks of length $Q$ is $n$.
        \item The first block, which is different in length than $Q$, has length $kQ+T$ for some $k\geq0$. Since $T<Q$, we can deduce $T$ from the length $kQ+T$.
        \item $P$ and $S$ are solutions of the Bezout's identity $SQ-TP=1$, with $P$ being the smallest positive integer.
        \item Now, from $K(\theta)$ we can deduce the M-sequence of the broken line $BL(\frac{P}{Q},\frac{a}{b},n)$. For each block of length $Q$ in $K(\theta)$ we put a word $W_{P/Q}$ in the M-sequence, and for each block of length $kQ+T$ we put the block $$W_{S/T}\underbrace{W_{P/Q}\cdots W_{P/Q}}_{k\text{ words}}.$$
        \item Finally, $a$ is the number of $1$'s in $BL(\frac{P}{Q},\frac{a}{b},n)$.
    \end{enumerate}
\end{itemize}
\end{proof}

\begin{example}
    \label{ExThmC}
    Let us consider the angle $$\theta_2=\tzo\left(\frac{2}{5},\frac{7}{17},2\right)=0.W_{2/5}^2\overline{W_{7/17}}=0.\overline{W_{2/5}W_{2/5}W_{1/2}W_{2/5}}.$$
    Part (a) of Theorem~\ref{thm:C} essentially translates the binary expansion of a broken line angle to its kneading sequence: for example, the $6^{\text{th}}$ and $11^{\text{th}}$ entries of $\theta_2$ given above are the first digits of strings $W_{2/5}W_{1/2}$ and $W_{1/2}$ respectively, therefore $K(\theta_2)=\overline{\mathnormal{1111011110111111}\ast}$.
    
    To illustrate part (b) of Theorem~\ref{thm:C}, suppose that $K(\theta)=\overline{\mathnormal{ 1111011110111101}\ast}$ is given and it is associated with a certain broken line $BL\left(\frac{P}{Q},\frac{a}{b},n\right)$. We determine the values of $\frac PQ, \frac ab$ and $n$ as follows: First, since the period of $K(\theta)$ is 17, then $b=17$. The length of the first block that ends in {${\mathfrak 0}$} is $5$, therefore, $Q=5$. There are 3 blocks of these types before a final block ending in $\ast$, so $n=3$. Next, the length of the final block is equal to $2=0Q+T$ so that $T=2$. Since $SQ-TP=1$, we solve the Bezout's identity $5S-2P=1$, with $P$ minimal, to get $S=1$ and $P=2$. By step 6 of the proof we have that the M-sequence of the broken line is $$0.\overline{W_{2/5}W_{2/5}W_{2/5}W_{1/2}},$$and therefore, by step 7, we have that $a=7$. We conclude that $\frac{P}{Q}=\frac{2}{5}$,$\frac{a}{b}=\frac{7}{17}$ and $n=3$.
\end{example}

\begin{corollary}\label{cor:Kconcat}
Fix $n\geq 1$ and let $\frac{P}{Q}, \frac{a}{b},\frac{S_n}{T_n}$ be rational numbers satisfying (\ref{eq:MainHyp}). Assume that $\frac{P}{Q}<\frac{p}{q}<\frac{a}{b}<\frac{s}{t}<\frac{S_n}{T_n}$ be rational numbers so that $\frac{a}{b}=\frac{p+s}{q+t}$. Denote the  kneading sequences associated with $BL(\frac{P}{Q},\frac{p}{q},n)$ and $BL(\frac{P}{Q},\frac{s}{t},n)$ by $\overline{\Omega_1\cdots\Omega_{q-1}*}$ and $\overline{\Gamma_1\cdots\Gamma_{t-1}*}$, respectively. Then the kneading sequence associated with $BL(\frac{P}{Q},\frac{a}{b},n)$ is
\[\overline{\Gamma_1\cdots\Gamma_{t-1}\mathnormal{1}\Omega_1\cdots\Omega_{q-1}*}.\]
\end{corollary}

\begin{proof}
If $0.\overline{\gamma_1\cdots\gamma_q}$ and $0.\overline{\beta_1\cdots\beta_t}$ denote the M-sequences associated with the broken lines $BL(\frac{P}{Q},\frac{p}{q},n)$ and $BL(\frac{P}{Q},\frac{s}{t}, n)$ respectively, then Proposition~\ref{prop:concat1} shows that the M-sequence associated with $BL(\frac{P}{Q},\frac{a}{b},n)$ is the concatenation $0.\overline{\beta_1\cdots\beta_t\gamma_1\cdots\gamma_q}$. The conclusion follows by applying Theorem~\ref{thm:C} to the binary expansion $0.\overline{\beta_1\cdots\beta_t\gamma_1\cdots\gamma_q}$.
\end{proof}

\subsection{Primitive components and localization}

Throughout this section, consider the Farey neighbors $0\leq \frac AB<\frac ST\leq 1$ and its mediant $0<\frac PQ<1$. We are now able to prove our last main result.

\medskip

\noindent
\textsc{Theorem A} (Primitive components). Fix $n\geq 1$ and let $\frac PQ, \frac ab$ be rational numbers satisfying  (\ref{eq:MainHyp}). If $\theta=\tzo(P/Q, a/b, n)$ is the M-sequence of the broken line $BL(\frac{P}{Q},\frac{a}{b},n)$ in the $01$-convention, then the external ray $R_\theta$ lands at the root of a primitive component of period $b$.

\begin{proof}
Corollary 5.5 in \cite{LS} essentially states that an external ray of angle $\theta$ lands at a root of a primitive component if and only if, pointwise $K^-(\theta)=\lim_{\alpha\nearrow\theta}K(\alpha)$ has minimal period $b$. On one hand, if $K(\theta)=\overline{\Omega_1\cdots\Omega_{b-1}*}$ is the kneading sequence associated with $\theta$, then the $01$-convention implies that $K^-(\theta)=\overline{\Omega_1\cdots\Omega_{b-1}\mathnormal{1}}$. Assume that $K^-(\theta)$ has period $k<b$, that is $b=mk$ for some $m>1$, so we can write
\[\Omega_1\cdots\Omega_{b-1}\mathnormal{1}=(\Omega_1\cdots\Omega_k)^m.\]
It follows that $\Omega_i=\Omega_{rk+i}$ for all $i=1,\ldots,k$ and $r=1,\ldots,m-1$.

The definition of a broken line at the $n^{\text{th}}$ hinge point shows that $\theta=0.W_{P/Q}^nW_{S/T}\cdots$, and that Theorem~\ref{thm:C} implies that $\Omega_1=\cdots=\Omega_{Q-1}=\mathnormal{1}$ and $\Omega_Q=\mathnormal{0}$. Therefore, $k\geq Q$ and
\begin{equation}\label{eq:k1}
\Omega_{rk+1}=\cdots=\Omega_{rk+Q-1}=\mathnormal{1}\qquad \text{ while }\qquad\Omega_{rk+Q}=\mathnormal{0},
\end{equation}
for all $r=0,\ldots,m-1$. This condition implies that $\alpha_{rk+1}\cdots\alpha_{rk+Q}=W_{P/Q}$ for all $r=0,\ldots,m-1$. Writing $\theta=0.\overline{W_{P/Q}^nW_{S/T}\cdots W_{S/T}W_{P/Q}^l}$ for some $l\geq0$, then Theorem~\ref{thm:C} shows that $\Omega_{b-lQ-T}=\mathnormal{0}$ and that $\Omega_{b-lQ-T+1}=\cdots=\Omega_b=\mathnormal{1}$. Therefore
\begin{equation}\label{eq:k2}
\Omega_{rk-lQ-T}=\mathnormal{0}\qquad \text{ while }\qquad \Omega_{rk-lQ-T+1}=\cdots=\Omega_{rk}=\mathnormal{1},
\end{equation}
for all $r=1,\ldots,m$. This last condition implies that $\alpha_{rk-lQ-T+1}\cdots\alpha_{rk}=W_{S/T}W_{P/Q}^l$ for all $r=1,\ldots,m$. Combining the implications of (\ref{eq:k1}) and (\ref{eq:k2}), it follows that
\[\alpha_{rk+1}\cdots\alpha_{(r+1)k}=W_{P/Q}\cdots W_{S/T}W_{P/Q}^l\]
for all $r=0,\ldots,m-1$. The identity $\Omega_i=\Omega_{rk+i}$ for all $i=1,\ldots,k$ and $r=1,\ldots,m-1$, shows that $\alpha_{rk+1}\cdots\alpha_{(r+1)k}=\alpha_1\cdots\alpha_k$ for all $r=1,\ldots,m-1$, therefore $0.\overline{\alpha_1\cdots\alpha_b}$ has period $k<b$ which is a contradiction. We conclude that $K^-(\theta)$ has minimal period $b$ and therefore, the external ray of angle $\theta$ lands at the root of a primitive component of the Mandelbrot set.

\end{proof}

We conclude this section by discussing the localization of primitive components of broken lines. Denote by $\omega=\omega_{P/Q}$ the \emph{junction point} of the principal antenna in the $P/Q$-limb. This is the unique parameter in the  limb for which the critical value of $f_\omega(z)=z^2+\omega$ maps into its $\alpha$-fixed point after $Q$ iterations. As shown in Proposition 3.2 in \cite{DM}, the Misiurewicz parameter $\omega$ is the landing point of $Q$ preperiodic external angles, denoted by $\eta_1,\ldots,\eta_Q$ that satisfy the order
\[\theta_{01}(P/Q)<\eta_1<\ldots<\eta_Q<\theta_{10}(P/Q).\]
Fix $1\leq j\leq Q-1$ and denote by $R_{\eta_j}$ the parametric external ray of angle $\eta_j$. Then, $\bbC \setminus (R_{\eta_j}\sqcup R_{\eta_{j+1}}\sqcup \omega)$ separates the complex plane into two sectors, $S_\heartsuit \sqcup S_j$ where $H_\heartsuit\subset S_\heartsuit$. The set $S_j\cap \mathcal{M}$ is called the $j^\text{th}$ \emph{spoke} of the main antenna. One can easily extend the definitions of junction points, antennas and spokes to any sublimb of a bulb $H(P/Q)$.

We will show that any primitive component with a characteristic angle given by a broken line $BL(P/Q, a/b, n)$ at the first hinge point (so $n=1$ and  $\frac PQ<\frac ab<\frac ST$) lies in the first spoke of the main antenna of the $P/Q$-limb. The locations of associated primitive components with $BL(P/Q, a/b, n)$ with $n\geq 2$ will be described via a generalization of several results in \cite{DM}. We begin with the following proposition.
\begin{proposition}\label{prop:p3.2}
Let $n\geq1$ and let $H$ be the satellite component of the bulb $H(P/Q)$ at internal angle $1/(n+1)$, that is $H=H(P/Q)\ast H(1/(n+1))$. Then, there is a parameter value $\omega=\omega(n,P/Q)$ in the antenna of the $\frac 1{n+1}$-sublimb associated with $H$ that has the following properties.
\begin{enumerate}
\item The orbit of $\omega$ under $f_\omega(z)=z^2+\omega$ lands on a repelling fixed point after exactly $Q^n$ iterations.
\item Two of the rays landing at $\omega$ have angles with binary expansions in the $01$-convention given by $0.W_{P/Q}^n\overline{W'_{P/Q}}$ and $0.W_{P/Q}^{n-1}W'_{P/Q}\overline{W_{P/Q}}$.
\end{enumerate} 
\end{proposition}

\begin{remark}
When $n=1$, the parameter $\omega$ above coincides with the junction point $\omega_{P/Q}$ given in Proposition 3.2 \cite{DM}.
\end{remark}

\begin{proof}
Fix $n\geq 1$ and consider the dyadic angle $1/2^n$. The parametric external ray associated with this angle lands at a Misiurewicz parameter $c_n$ and the orbit of the critical value under $f_{c_n}$ lands at the $\beta_{c_n}$-fixed point after $n$ iterations. The angle $1/2^n$ has two binary expansions given by $0.0^{n-1} 0\overline{1}$ and $0.0^{n-1} 1\overline{0}$. From Proposition~\ref{prop:ExtAngles} it follows that $\{\theta,\theta'\}=\{\tzo(P/Q), \toz(P/Q)\}=\{0.\overline{W_{P/Q}}, 0.\overline{W'_{P/Q}}\}$ is the characteristic pair of $H(P/Q)$. Applying Douady's tuning procedure, the parametric external rays with angles
\begin{eqnarray*}
\varphi_{1} :=& 0.0^{n-1}0\overline{1} \ast H(P/Q)
= 0.W^{n-1}_{P/Q} W_{P/Q}\overline{W'_{P/Q}}, \\
\varphi_{n}:=& 0.0^{n-1}1\overline{0} \ast H(P/Q)
= 0. W_{P/Q}^{n-1} W'_{P/Q} \overline{W_{P/Q}},
\end{eqnarray*}
land at a point $\omega$ in the $P/Q$-limb (and if fact, $\omega$ must lie in the $\frac 1{n+1}$-sublimb associated with $H$ as $c_n$ lies at a tip of the antenna associated with the bulb $H(\frac{1}{n+1})$). Since $n\geq 1$ and the words $W_{P/Q}$ and $W_{P/Q}'$ differ exactly in their last two digits, it follows that both $\varphi_1$ and $\varphi_n$ have strict preperiod $nQ$ and period $Q$.

The same arguments given in Proposition 3.2 in \cite{DM} show that the critical value under $f_\omega$ lands at the $\alpha_\omega$-fixed point after exactly $nQ$ iterations, and $\alpha_\omega$ is the landing point of the dynamical external rays of angles $0.\overline{W_{P/Q}}$ and $0.\overline{W'_{P/Q}}$. As established in \cite[Expos\'e VIII]{DH1},  there are $Q$ periodic rays landing at $\alpha_\omega$. Since $f_\omega^{-nQ}$ acts as a homeomorphism between small neighborhoods of $\alpha_\omega$ and $\omega$, we obtain the existence of $Q$ preperiodic rays landing at $\omega$ (both in dynamic and parametric planes), and in particular, the rays of angles $\varphi_1$ and $\varphi_n$ land at $\omega$.
\end{proof}

Using Proposition 5.2 in \cite{DM} and a similar argumentation as in Theorem 5.3 in \cite{DM}, one can provide the explicit binary expansions of the angles associated with the $Q$ preperiodic rays landing at $\omega$ as follows.

\begin{proposition}\label{prop:rays-01}
Let $\omega=\omega(n,P/Q)$ as in Proposition \ref{prop:p3.2}. Then there exist $Q$ preperiodic rays contained in the interval $[0.W_{P/Q}^n\overline{W_{P/Q}'},0.W_{P/Q}^{n-1}W_{P/Q}'\overline{W_{P/Q}}]$ that land on $\omega$. These rays are given in increasing order by $ \varphi_1< \varphi_2<\ldots <\varphi_n$, where
\begin{equation} \label{eq:vrays}
\varphi_k= \begin{cases}
0.W_{P/Q}^{n-1} W_{P/Q} \sigma^{(k-1)B}(\overline{W'_{P/Q}}) & \textrm{if $1\leq k\leq Q-P$}\\
0.W_{P/Q}^{n-1} W'_{P/Q} \sigma^{(k-1)B}(\overline{W'_{P/Q}})& \textrm{ if $Q-P+1\leq k\leq Q$,}
\end{cases}
\end{equation}
where $\sigma$ denotes the shift map acting over binary sequences.
\end{proposition}

The following result provides the location of a primitive component associated with a broken line $BL(P/Q, a/b, n)$. 

\begin{proposition}\label{prop:location}
Fix $n\geq 1$ and select $\frac PQ, \frac ab$ as in (\ref{eq:MainHyp}). Let $\theta$ denote the binary expansion of the broken line $BL(P/Q,a/b,n)$ in the $01$-convention. Then $\varphi_1<\theta<\varphi_2$, that is, the primitive component with external ray $\theta$ lies in the first spoke of the point $\omega(n,Q)$.
\end{proposition}
\begin{proof}
We first show that $\varphi_1<\theta$. Observing that the preperiodic part of $\varphi_1=0.W^{n}_{P/Q}\overline{W'_{P/Q}}$ coincides with the preperiodic part of $\theta=0.W^n_{P/Q}\overline{W_{a/b}}$, it is enough to compare the periodic parts of each binary expression. By hypothesis, $0<\frac PQ<\frac ab<1$, hence  the characteristic rays of the bulb $H(P/Q)$ have angles strictly smaller than the angles of the characteristic rays of $H(a/b)$, so in particular $0.\overline{W'_{P/Q}}<0.\overline{W_{a/b}}$ and therefore $\varphi_1<\theta$.

To prove the second inequality, assume first the case when $Q-P\geq 2$ so that $\varphi_2=0.W_{P/Q}^n\sigma^{B}(\overline{W'_{P/Q}})$. Using Corollary \ref{cor:Cor1} write $\theta=0.W^n_{P/Q}\overline{W_{a/b}}= 0.W_{P/Q}^n\overline{W_{S/T}\ldots W_{P/Q}}$. If  $\frac ST \neq 1$ and $\frac AB \neq 0$ (so that $W'_{S/T}$ and $W'_{A/B}$ are well-defined), then we obtain 
\[0.\overline{W_{a/b}}=0.\overline{W_{S/T}\ldots W_{P/Q}} <0.\overline{W'_{S/T}W'_{A/B}}=0.\sigma^B(\overline{W'_{P/Q}}),\]
and therefore $\theta < \varphi_{2}$. If $\frac AB=0$ and $\frac ST<1$ as in Remark \ref{rem:RLemPEsp1}, then for some $N\geq 2$, $\frac ST=\frac 1N, \frac PQ = \frac{1}{N+1}$ and $W_{S/T}=0^{N-1} 1$ whereas $W_{P/Q} = 0^{N}1$. The periodic parts of the angles now become
\[
0.\overline{W_{a/b}}=0.\overline{W_{S/T}\cdots W_{P/Q}}\in [0.0^{N-1} 1] <[0.0^{N-2}100]\ni 0.\sigma(\overline{W'_{P/Q}}),\]
and once again $\theta<\varphi_2$. The case $\frac AB\geq 0, \frac ST=1$, as discussed in Remark \ref{rem:RLemPEsp2}, shows that for some $N\geq 0$, one has $\frac AB = \frac{N}{N+1}, \frac PQ = \frac{N+1}{N+2}$, which is the remaining case in which $Q-P=1$. This case falls into the second condition in (\ref{eq:vrays}) so that
\[\varphi_2=0.W_{P/Q}^{n-1}W'_{P/Q} \sigma^B(\overline{W'_{P/Q}}).\]
In any case, if $\varphi_2$ is given as above, a direct comparison of the preperiodic parts of $\theta$ and $\varphi_2$ shows the sought inequality.

\end{proof}

\appendix
\section{Appendix: The 10-convention}\label{App:A}

We present the corresponding versions of our most important results with respect to the $10$-convention. Once the appropriate hypotheses given below are taken into consideration, then the proofs of these results are completely analogous to proofs in the $01$-convention, therefore they will be omitted.

Let $n\geq 1$ and set  $\theta=\toz(\frac PQ,\frac a b,n)$, with fractions $\frac PQ, \frac ab, \frac AB, \frac ST$ satisfying the \emph{10-hypothesis}: namely $\frac AB, \frac ST$ are Farey neighbors,
\begin{equation}\label{eq:MainHyp10}
0\leq \frac{A}{B}<\frac{A_n}{B_n}< \frac ab<\frac{P}{Q}<\frac{S}{T}\leq 1,\quad \frac PQ = \frac{A+S}{B+T},\quad\text{and}\quad \frac{A_n}{B_n}=\frac{A+(n-1)P}{B+(n-1)Q}. \tag{10-Hyp}
\end{equation}

\begin{lemma}
(Lemma~\ref{lem:persturm} in 10-convention)\label{lem:persturm10}
Fix $n\geq 1$ and let $\frac AB,\frac PQ$ satisfy (\ref{eq:MainHyp10}). If $a/b$ is selected so that $\frac{A_n}{B_n} < \frac ab < \frac PQ$, then $\toz(P/Q, a/b, n)$ is the binary expansion of a periodic Sturmian angle of period $b$.
\end{lemma}

Proposition~\ref{prop:P1} and Corollary~\ref{cor:Cor1} describe the correct order of the concatenation of M-sequences of Farey neighbors to obtain M-sequences of their Farey descendants. We will omit the $10$ superscripts in the wordwise expression of M-sequences and simply write $\toz(p/q)=0.\overline{W_{p/q}}$. We begin by describing the concatenation order for M-sequences of broken lines.

\begin{proposition}(Proposition~\ref{prop:concat1} in 10-convention) \label{prop:concat1-10}
Let $a/b$ and $c/d$ be Farey neighbors and $f/g$ their mediant, so that $\frac{A_n}{B_n}<\frac{a}{b}<\frac{f}{g}<\frac{c}{d}\leq\frac{P}{Q}$. If the M-sequences of the broken lines $BL(\frac PQ,\frac a b,n)$ and $BL(\frac PQ,\frac c d,n)$ are $0.\overline{\alpha_1\cdots\alpha_b}$ and $0.\overline{\beta_1\cdots\beta_d}$ respectively, then the M-sequence of $BL(\frac PQ,\frac f g,n)$ is $0.\overline{\alpha_1\cdots\alpha_b\beta_1\cdots\beta_d}$.
\end{proposition}

For $m\geq 0$ consider the rational numbers
\begin{equation}\label{eq:AmBm}
\frac {A_n}{B_n}< \frac{P_m}{Q_m}\leq \frac PQ\qquad\text{where}\qquad \frac{P_m}{Q_m}:=\frac{P+mA_n}{Q+mB_n}.
\end{equation}

\begin{proposition}(Proposition~\ref{prop:SnTn} in 10-convention)\label{prop:SnTn-10}
Let $\frac{A_n}{B_n}<\frac{P_m}{Q_m}< \frac{P}{Q}$ for some $n\geq 1$ and $m\geq 1$. Then
\[\toz\left(\frac P Q,\frac{a}{b}, n\right)=0.\overline{W_{P/Q}^n(W_{A/B}W_{P/Q}^{n-1})^{m-1}W_{A/B}}.\]
\end{proposition}

\begin{remark} \label{rem:R1-10}
For $n,m\geq 1$ define the block
\begin{equation}\label{def:D1-10}
B_{n,m}:=W_{P/Q}^n(W_{A/B}W_{P/Q}^{n-1})^{m-1}W_{A/B},
\end{equation}
and for $m=0$, set $B_{n,0}:=W_{P/Q}$. Then, if $\frac ab$ is a fraction so that $\frac{P_{m+1}}{Q_{m+1}} < \frac ab < \frac{P_m}{Q_m}$ for some integer $m\geq0$, then
\[\toz \left(\frac PQ,\frac ab,n\right)=0.\overline{B_{n,m_1}B_{n,m_2}\cdots B_{n,m_{k-1}}B_{n,m_k}}=0.\overline{B_{n,m+1}B_{n,m_2}\cdots B_{n,m_{k-1}}B_{n,m}},\]
for some $k\geq2$ and $m_i\in\{m,m+1\}$ for all $1\leq i\leq k$. In particular if $\frac{A_n}{B_n}<\frac ab<\frac PQ$, then the M-sequence of the broken line $BL(\frac PQ,\frac ab, n)$ is given by
\begin{equation}\label{eq:bl-10}
\theta:=\toz\left(\frac PQ,\frac ab, n\right)=0.\overline{B_{n,m_1}\cdots B_{n,m_k}},
\end{equation}
for some $k\geq1$ and $m_i\in\{m,m+1\}$ for some $m\geq0$. 
\end{remark}

\subsection{Characteristic external angles}

\begin{definition}(Definition~\ref{def:D2} in 10-convention)\label{def:D2-10}
Given a finite word $W=\alpha_1\cdots\alpha_n$ with $\alpha_j\in\{0,1\}$ for $j=1,\ldots,n$, denote by $W'=a_1\cdots a_n$ the word that satisfies the equation $W=W'+1$, that is, $W'=W-1$.
\end{definition}

\begin{definition}(Definition~\ref{def:tetas} in 10-convention)\label{def:tetas-10}
Fix $n\geq 1$, consider the fractions $\frac{a}{b},\frac{P}{Q}$ as in (\ref{eq:MainHyp10}) and let $\theta:=\toz (\frac PQ,\frac ab, n)$ be the M-sequence (with $10$-convention) of the broken line $BL(\frac PQ, \frac ab, n)$. If $\theta=0.\overline{B_{n,m_1}\cdots B_{n,m_k}}$ then define 
\[\theta':=0.\overline{B_{n,m_1}'\cdots B_{n,m_k}'}.\]
Writing the above expressions digitwise as $\theta=0.\overline{\alpha_1\cdots\alpha_b}$ and $\theta'=0.\overline{\alpha_1'\cdots\alpha_b'}$, define for each $k=1,\ldots,b$, 
\begin{equation}\label{eq:tetak-10}
\theta_k:=0.\alpha_{b-k+1}'\cdots\alpha_b'\overline{\alpha_1\cdots\alpha_b}.
\end{equation}
\end{definition}

From now on, we adopt the convention $W_{0/1}^{10}:=0$.

\begin{lemma}(Lemma~\ref{lem:Lem2} in 10-convention)\label{lem:Lem2-01}
Let $\frac{a}{b}$ and $\frac{c}{d}$ be Farey neighbors such that $0\leq a/b<c/d< 1$ and let $p/q$ be their mediant. Given the associated words $W_{a/b}$, $W_{c/d}$ and $W_{p/q}$, one has
\[W_{p/q}=W_{a/b}W_{c/d}=W_{c/d}'W_{a/b}.\]
\end{lemma}

\begin{lemma}(Lemma~\ref{lem:LemP} in 10-convention)\label{lem:LemP-10}
Fix $n\geq 1$ and let $\frac PQ, \frac{a}{b}$ be rational numbers satisfying (\ref{eq:MainHyp10}). Additionally assume that $0<\frac AB<\frac ST<1$. Let 
$\theta$, $\theta'$ and $\theta_k$ be given as in Definition~\ref{def:tetas-10}. For any $k=2,\ldots,b$,
\begin{enumerate}
\item if $\alpha_{b-k+2}$ is the first digit of $W_{P/Q}^sW_{A/B}$ for some $0\leq s<n$, then
\[\theta_k<2^{b-k}\theta\mod 1<2^{b-1}\theta\mod 1<\theta_1,\]
\item otherwise, $2^{b-1}\theta\mod 1<\theta_k,2^{b-k}\theta\mod 1<\theta_1$.
\end{enumerate}
In both cases, the intervals $(\theta_1,2^{b-1}\theta\mod 1)$ and $(2^{b-k}\theta\mod 1,\theta_k)$ are unlinked.
\end{lemma}

\begin{lemma}(Lemma~\ref{lem:LemPEsp1} in 10-convention)\label{lem:LemPEsp1-10}
Fix $n\geq 1$ and let $\frac PQ, \frac{a}{b}$ be rational numbers satisfying (\ref{eq:MainHyp10}). Additionally assume that $0<\frac AB<\frac ST=1$. Let $\theta$, $\theta'$ and $\theta_k$ be given as in Definition~\ref{def:tetas-10}. For any $k=2,\ldots,b$,
\begin{enumerate}
\item if $\alpha_{b-k+2}$ is the first digit of $W_{P/Q}^sW_{S/T}$ for some $0\leq s<n$, then $$\theta_k<2^{b-k}\theta<2^{b-1}\theta<\theta_1,$$
\item otherwise, we have $2^{b-1}\theta<\theta_k,2^{b-k}\theta<\theta_1$.
\end{enumerate}
In both cases, the intervals $(\theta_1,2^{b-1}\theta)$ and $(2^{b-k}\theta,\theta_k)$ are unlinked.
\end{lemma}

\begin{lemma}(Lemma~\ref{lem:LemPEsp2} in 10-convention)
Fix $n\geq 1$ and let $\frac PQ, \frac{a}{b}$ be rational numbers satisfying (\ref{eq:MainHyp10}). Additionally assume that $0=\frac{A}{B}<\frac{S}{T}\leq1$. Let $\theta$, $\theta'$ and $\theta_k$ be given as in Definition~\ref{def:tetas-10}. For any $k=2,\ldots,b$,
\begin{enumerate}
\item if $\alpha_{b-k+2}$ is the first digit of $W_{P/Q}^sW_{A/B}$ for some $0\leq s<n$, then 
\begin{enumerate}
    \item if $n=1$ then $2^{b-k}\theta<2^{b-1}\theta<\theta_1<\theta_k$.
    \item if $n>1$ and $\alpha_{b-k+2}$ corresponds to the last word $W_{A/B}$ of some block $B_{n,m}$, for some $m\geq1$, then $2^{b-k}\theta<2^{b-1}\theta<\theta_1<\theta_k$.
    \item if $n>1$ and $\alpha_{b-k+2}$ does not correspond to the last word $W_{A/B}$ of some block $B_{n,m}$, for some $m\geq1$, then $\theta_k<2^{b-k}\theta<2^{b-1}\theta<\theta_1.$
\end{enumerate}
\item otherwise, we have $2^{b-1}\theta<\theta_k,2^{b-k}\theta<\theta_1$.
\end{enumerate}
In both cases, the intervals $(\theta_1,2^{b-1}\theta)$ and $(2^{b-k}\theta,\theta_k)$ are unlinked.
\end{lemma}

Combining the above lemmas we can derive the first main result in the Appendix.
 
\begin{theorem}(Theorem~\ref{thm:B} in 10-convention)\label{thm:TheoremB-10}
Fix $n\geq 1$ and consider the fractions $\frac PQ, \frac ab$ satisfying (\ref{eq:MainHyp10}). Then $\theta$ and $\theta'$ are conjugate external angles.
\end{theorem}

\subsection{Kneading sequences and primitive components}

\begin{theorem}(Theorem~\ref{thm:C} in 10-convention)
 Fix $n\geq 1$ and let $\frac{A}{B},\frac{a}{b},\frac{P}{Q}$ be rational numbers satisfying (\ref{eq:MainHyp10}). Denote the M-sequence of the broken line $BL(\frac{P}{Q},\frac{a}{b},n)$ by $\theta=0.\overline{\alpha_1\cdots\alpha_b}= 0.\overline{W_{\zeta_1}\ldots W_{\zeta_k}}$ for some $k\geq1$ and $\zeta_i\in\{A/B,P/Q\}$ for all $i=1,\ldots,k$.  Let $K(\theta)=\overline{\Omega_1\cdots\Omega_{b-1}*}$ denote the kneading sequence associated with $\theta$, with $\Omega_i\in \{\mathnormal{0,1}\}$ for all $i=1,\ldots,b-1$. Then, 
\begin{itemize}
    \item[(a)] $\Omega_i=\mathnormal{0}$ if and only if $\alpha_{i+1}$ is the first digit of the string $W_{P/Q}^sW_{A/B}$ for some $0\leq s<n$.
    \item[(b)] The numbers $\frac{P}{Q}$, $\frac{A}{B}$, $\frac{a}{b}$ and the M-sequence of $BL(\frac{P}{Q},\frac{a}{b},n)$ can be recovered from the kneading sequence $K(\theta)$.
\end{itemize}
\end{theorem}

\begin{corollary}(Corollary~\ref{cor:Kconcat} in 10-convention)
Fix $n\geq 1$, let $a/b$ and $c/d$ be Farey neighbors and $f/g$ their mediant, so that $\frac{A_n}{B_n}<\frac{a}{b}<\frac{f}{g}<\frac{c}{d}\leq\frac{P}{Q}$. Denote the  kneading sequences associated with $BL(\frac{P}{Q},\frac ab,n)$ and $BL(\frac{P}{Q},\frac cd,n)$ by $\overline{\Omega_1\cdots\Omega_{q-1}*}$ and $\overline{\Gamma_1\cdots\Gamma_{t-1}*}$, respectively. Then the kneading sequence associated with $BL(\frac{P}{Q},\frac fg,n)$ is
\[\overline{\Omega_1\cdots\Omega_{q-1}\mathnormal 1\Gamma_1\cdots\Gamma_{t-1}*}.\]
\end{corollary}

\begin{theorem}(Theorem~\ref{thm:A} in 10-convention)\label{thm:RootP-10}
Fix $n\geq 1$ and let $\frac PQ, \frac ab$ be rational numbers satisfying (\ref{eq:MainHyp10}). If $\theta=\toz(P/Q,a/b,n)$ is the M-sequence of the broken line $BL(\frac{P}{Q},\frac{a}{b},n)$ in the $10$-convention, then the external ray $R_\theta$ lands at the root of a primitive component of period $b$.
\end{theorem}

The localization of the primitive component given in the above theorem is derived from the following results.

\begin{proposition}(Proposition~\ref{prop:p3.2} in 10-convention)\label{prop:p3.2-10}
Let $n\geq1$ and let $H$ be the satellite component of the bulb $H(P/Q)$ at internal angle $n/(n+1)$, that is $H=H(P/Q)\ast H(n/(n+1))$. Then, there is a parameter value $\omega=\omega(n,P/Q)$ in the antenna of the $\frac n{n+1}$-sublimb associated with $H$ that has the following properties.
\begin{enumerate}
\item The orbit of $\omega$ under $f_\omega(z)=z^2+\omega$ lands on a repelling fixed point after exactly $Q^n$ iterations.
\item Two of the rays landing at $\omega$ have angles with binary expansions given by $0.(W_{P/Q})^{n-1}$ $W'_{P/Q} \overline{W_{P/Q}}$ and $0.(W_{P/Q})^n\overline{W'_{P/Q}}$.
\end{enumerate} 
\end{proposition}

\begin{proposition}(Proposition~\ref{prop:rays-01} in 10-convention)
Let $\omega=\omega(n,P/Q)$ as in Proposition \ref{prop:p3.2-10}. Then there exist $Q$ preperiodic rays contained in the interval $[0.(W_{P/Q})^{n-1}W'_{P/Q}\overline{W_{P/Q}},$ $0.(W_{P/Q})^n\overline{W'_{P/Q}}]$ that land on $\omega$. These rays are given in increasing order by $ \varphi_1< \varphi_2<\ldots <\varphi_Q$, where
\begin{equation*}
\varphi_k= \begin{cases}
0.(W_{P/Q})^{n-1} W'_{P/Q} \sigma^{(k-1)B}(\overline{W_{P/Q}}) & \textrm{if $1\leq k\leq Q-P$}\\
0.(W_{P/Q})^{n-1} W_{P/Q} \sigma^{(k-1)B}(\overline{W_{P/Q}})& \textrm{ if $Q-P+1\leq k\leq Q$.}
\end{cases}
\end{equation*}
\end{proposition}

\begin{proposition}(Proposition~\ref{prop:location} in 10-convention)\label{prop:location-10}
Fix $n\geq 1$ and select $\frac PQ, \frac ab$ as in (\ref{eq:MainHyp10}). Let $\theta$ denote the binary expansion of the broken line $BL(P/Q,a/b,n)$ in the $10$-convention. Then $\varphi_{Q-1}<\theta<\varphi_Q$, that is, the primitive component with external ray $\theta$ lies in the $(Q-1)^{\text{th}}$ spoke of the point $\omega(n,Q)$.
\end{proposition}

\bibliographystyle{alpha}
\bibliography{simple}

\end{document}